\documentclass[a4paper,UKenglish,cleveref, autoref, thm-restate]{lipics-v2021}

\usepackage{amsmath}
\usepackage{amssymb}
\usepackage{amsfonts,amsthm}
\usepackage{thmtools,thm-restate}
\usepackage{tabularx,lipsum,environ}
\usepackage{multirow}
\usepackage{hyperref}
\usepackage{graphicx,float}
\usepackage{linegoal}
\usepackage{caption}
\usepackage{subcaption}
\usepackage{complexity}

\newcommand{\set}[1]{\ensuremath{ \left\lbrace  #1 \right\rbrace }}
\newcommand{\app}{$\spadesuit$}

\usepackage{tikz}
\usetikzlibrary{arrows}
\usetikzlibrary{arrows.meta}
\usetikzlibrary{decorations.pathmorphing}
\usetikzlibrary{decorations.markings}
\usetikzlibrary{calc}
\tikzset{
  circ/.style = {circle,draw,fill,inner sep=1.3pt},
  circr/.style = {circle,draw=red,fill=red,inner sep=1.3pt},
  scirc/.style = {circle,draw,fill,inner sep=.5pt},
  invisible/.style = {draw=none,inner sep=0pt,font=\small},
  nonedge/.style={decorate,decoration={snake,amplitude=.3mm,segment length=1mm},draw}
}

\newcommand{\longPaw}[5]{
\node[circ,label=below:{\tiny #4}] (a) at (#1+0,#2) {};
\node[circ,label=below:{\tiny #5}] (b) at (#1+1,#2) {};
\node[circ] (c) at (#1+.5,#2+.3) {};
\node[circ] at (#1+.5,#2+.6) {};
\node[circ] (d) at (#1+.5,#2+.9) {};
\node[draw=none] at (#1+1,#2+.6) {\small \textbf{#3}};

\draw (a) -- (b)
(a) -- (c) 
(b) -- (c) 
(c) -- (d);
}

\newcommand\yes{\textsc{Yes}}
\newcommand\no{\textsc{No}}

\newcommand\contracd{\textsc{1-Edge Contraction($\gamma$)}\,}
\newcommand\contractd{\textsc{1-Edge Contraction($\gamma_t$)}\,}
\newcommand\contracstd{\textsc{1-Edge Contraction($\gamma_{t2}$)}\,}

\title{Using edge contractions to reduce the semitotal domination number} 

\author{Esther Galby}{CISPA Helmholtz Center for Information Security, Saarbr\"ucken, Germany}{esther.galby@cispa.saarland}{}{Research supported by the European Research Council (ERC) consolidator grant No. 725978 SYSTEMATICGRAPH.}
\author{Paloma T. Lima}{University of Bergen, Bergen, Norway}{paloma.lima@uib.no}{}{}
\author{Felix Mann}{University of Fribourg, Fribourg, Switzerland}{felix.mann@unifr.ch}{}{}
\author{Bernard Ries}{University of Fribourg, Fribourg, Switzerland}{bernard.ries@unifr.ch}{}{}

\authorrunning{E. Galby and P.\,T. Lima and F. Mann and B. Ries } 

\Copyright{Esther Galby and Paloma T. Lima and Felix Mann and Bernard Ries} 

\ccsdesc[100]{Mathematics of computing~Graph theory}

\keywords{Blocker problem, Edge contraction, Semitotal domination} 

\nolinenumbers 

\EventEditors{John Q. Open and Joan R. Access}
\EventNoEds{2}
\EventLongTitle{42nd Conference on Very Important Topics (CVIT 2016)}
\EventShortTitle{CVIT 2016}
\EventAcronym{CVIT}
\EventYear{2016}
\EventDate{December 24--27, 2016}
\EventLocation{Little Whinging, United Kingdom}
\EventLogo{}
\SeriesVolume{42}
\ArticleNo{23}

\begin{document}

\maketitle

\begin{abstract}
In this paper, we consider the problem of reducing the semitotal domination number of a given graph by contracting $k$ edges, for some fixed $k \geq 1$. We show that this can always be done with at most 3 edge contractions and further characterise those graphs requiring 1, 2 or 3 edge contractions, respectively, to decrease their semitotal domination number. We then study the complexity of the problem for $k=1$ and obtain in particular a complete complexity dichotomy for monogenic classes. 
\end{abstract}


\section{Introduction}
\label{sec:intro}

In the standard graph modification problem, one is interested in modifying a given graph, using a minimum number of graph operations from a prescribed set, such that the resulting graph belongs to some fixed graph class. The related family of so-called \textit{blocker problems} considers not a graph class but rather asks for a specific graph parameter $\pi$ to decrease: given a graph $G$, a set $\mathcal{O}$ of one or more graph operations and an integer~$k\geq 1$, the question is whether $G$ can be transformed into a graph $G'$ by using at most $k$ operations from $\mathcal{O}$ such that $\pi(G')\leq \pi(G)-d$ for some {\it threshold} $d\geq 1$. These types of problems (as well as the variants where one wants to increase some parameter $\pi$) are related to other well-known graph problems like for instance \textsc{Hadwiger Number}, \textsc{Club Contraction} and \textsc{Graph Transversal} (see \cite{DPPR15,PPR16}), and have been extensively studied in the literature (see, e.g.,  \cite{bazgan2013critical,Bentz,CHEN,CWP11,DPPR15,diner2018contraction,domcontract,P3kP2,GAL-MAN-RIE2,MFCS2020,PBP,PPR16,paulusma2017blocking,paulusma2018critical,rautenbach,ZENKLUSEN2,ZENKLUSEN1}). Furthermore, identifying the part of a graph which makes a certain parameter increase or decrease significantly may give important information about the structure of the graph.

In this paper, we focus on one particular graph operation, namely \textit{edge contraction}. Contracting an edge $uv$ in a graph $G$ corresponds to deleting both vertices $u$ and $v$ and adding a new vertex which is made adjacent to every neighbour of $u$ or $v$ in the original graph $G$. We denote by $ct_\pi(G)$ the smallest integer $k$ such that there is a set of $k$ edges in $E(G)$ whose contraction yields a graph for which the value of $\pi$ is strictly smaller than $\pi(G)$. 

In \cite{HX10}, Huang and Xu considered for $\pi$ the \emph{domination number} (denoted by $\gamma$) and the \emph{total domination number} (denoted by $\gamma_t$). They showed that for $\pi \in \{\gamma,\gamma_t\}$, $ct_\pi(G)$ is never greater than $3$ and further characterised for every fixed $k \in \{1,2,3\}$, the graphs for which $ct_\pi(G)=k$ in terms of the structure of their (total) dominating sets. More specifically, they showed the following (see \Cref{sec:prelim} for missing definitions).


\begin{theorem}[\cite{HX10}]\label{theorem:contracdom}
For any graph $G$, the following holds.
\begin{itemize}
\item[(i)] $ct_\gamma (G)=1$ if and only if there exists a minimum dominating set in $G$ that is not independent.
\item[(ii)] $ct_\gamma (G)=2$ if and only if every minimum dominating set in $G$ is independent and there exists a dominating set $D$ in $G$ of size $\gamma(G)+1$ such that $G[D]$ contains at least two edges.
\end{itemize}
\end{theorem}

\begin{theorem}[\cite{HX10}]\label{theorem:1totcontracdom}
For any graph $G$, the following holds.
\begin{itemize}
\item[(i)] $ct_{\gamma_t} (G)=1$ if and only if there exists a minimum total dominating set $D$ in $G$ such that $G[D]$ contains a $P_3$.
\item[(ii)]  $ct_{\gamma_t} (G)=2$ if and only if every minimum total dominating set in $G$ induces a graph that does not contain a $P_3$ and there exists a total dominating set $D$ in $G$ of size $\gamma_t(G)+1$ such that $G[D]$ contains a subgraph isomorphic to $P_4$, $K_{1,3}$ or $2P_3$.
\end{itemize}
\end{theorem}

In this paper, we consider the \emph{semitotal domination number} (denoted by $\gamma_{t2}$): a \emph{semitotal dominating set} of a graph $G$ is a set $D\subseteq V(G)$ such that every vertex in $V(G) \setminus D$ has a neighbour in $D$ (that is, $D$ is dominating) and every vertex in $D$ is at distance at most two from another vertex in $D$; and the \emph{semitotal domination number} of $G$ is the size of a minimum semitotal dominating set in $G$. We are more precisely interested in the following problem with $\pi = \gamma_{t2}$ and $k =1$. 

\begin{center}
\fbox{
\begin{minipage}{5.5in}
\textsc{$k$-Edge Contraction($\pi$)}
\begin{description}
\item[Instance:] A graph $G$.
\item[Question:] Is $ct_{\pi}(G)\leq k$?
\end{description}
\end{minipage}}
\end{center}

Similarly to the above results, we show that $ct_{\gamma_{t2}}(G)\leq 3$ for every graph $G$ and further characterise for every fixed $k \in \{1,2,3\}$, those graphs for which $ct_{\gamma_{t2}}(G)=k$ in terms of the structure of their semitotal dominating sets. Let us note that the critical substructures are more complex and diverse than in the case of the (total) domination number (see Figure~\ref{fig:conf}). We then determine the computational complexity of \contracstd for several graph classes, such as bipartite graphs and chordal graphs, as well as for every monogenic graph class, that is, the set of $H$-free graphs for some fixed graph $H$. From these results, we deduce in particular the following theorem. 

\begin{theorem}
\label{thm:dichotomy3}
\contracstd{} is polynomial-time solvable for $H$-free graphs if $H$ is an induced subgraph of $P_5 + tK_1$ with $t \geq 0$ or $H$ is an induced subgraph of $P_3 + pK_2 + tK_1$ with $p,t\geq 0$, and $\mathsf{(co)NP}$-hard otherwise.
\end{theorem}

It has been shown in \cite{semitot}, that the complexities of the \textsc{Dominating Set} problem (that is, given a graph $G$ and an integer $k\geq 0$, does there exist a dominating set of size at most $k$?), the \textsc{Total Dominating Set} problem (given a graph $G$ and an integer $k\geq 0$, does there exist a total dominating set of size at most $k$?) and the \textsc{Semitotal Dominating Set} problem (given a graph $G$ and an integer $k\geq 0$, does there exist a semitotal dominating set of size at most $k$?) agree on all monogenic graph classes. One may therefore ask whether this is still true when we consider blocker problems with respect to these parameters together with edge contractions. Interestingly, this is no longer the case: combining our results with the complexity dichotomies for \contracd and \contractd obtained in \cite{domcontract,P3kP2} and \cite{GAL-MAN-RIE2}, respectively, we can observe that the complexities of \contracstd and \contractd disagree on some monogenic graph classes. Whether there is a hereditary graph class on which \contracd and \contracstd differ remains an open question; we note however that if such a class exists, its characterising set of forbidden induced subgraphs has to contain at least two graphs and in light of Lemma \ref{lem:cycles}, we conjecture that such a graph class has to have a graph non-isomorphic to a cycle as a forbidden induced subgraph. 

The paper is organised as follows\footnote{Proofs marked with \app~have been placed in the appendix.}. In Section \ref{sec:prelim}, we present definitions and notations that are used throughout the paper. Section~\ref{sec:struct} is devoted to the proofs of our structural results which we need for the remainder of the paper. In Section \ref{sec:probl} we consider different graph classes and determine the complexity of \contracstd in these classes. We then combine these results in Section~\ref{sec:main} to prove our main result, that is, Theorem~\ref{thm:dichotomy3}. 

\section{Preliminaries}
\label{sec:prelim}

Unless specified otherwise, we only consider finite, simple, connected graphs. For a graph $G$ we denote its vertex set by $V(G)$ and its edge set by $E(G)$. For a set $S\subseteq V(G)$, we let $G[S]$ denote the graph \emph{induced} by $S$, that is, the graph with vertex set $S$ and edge set $\set{xy\in E(G)\colon\, x,y\in S}$. For an edge $xy \in E(G)$, we denote by $G/xy$ the graph obtained from $G$ by contracting the edge $xy$. We say that two vertices $x$ and $y$ are \emph{adjacent} or \emph{neighbours} if $xy$ is an edge. The \emph{neighbourhood} $N(v)$ of a vertex $v\in V(G)$ is the set $\set{w\in V(G)\colon\, vw\in E(G)}$ and the \emph{closed neighbourhood} $N[v]$ of $v$ is the set $N(v)\cup\set{v}$. For any two vertices $v,w\in V(G)$, a vertex $u\in N(v)\cap N(w)$ is called a \emph{common neighbour} of $v$ and $w$. Given two sets $S,S'\subseteq V(G)$, we say that $S$ is \emph{complete} to $S'$ if every vertex in $S$ is adjacent to every vertex in $S'$. We say a vertex $v\in V(G)$ is complete to a set $S\subseteq V(G)$ if $\set{v}$ is complete to $S$. For any two vertices $x,y\in V(G)$, the \emph{distance between $x$ and $y$} is the number of edges in a shortest path from $x$ to $y$ and is denoted $d_G(x,y)$ (if it is clear from the context, the index may be omitted). A set $D\subseteq V(G)$ is a \emph{dominating set} if every vertex $v\in V(G)\setminus D$ has a neighbour in $D$. 

Let $D$ be a dominating set of $G$ and $w\in V(G)\setminus D$. For any neighbour $v\in D\cap N(w)$, we say that $v$ \emph{dominates} $w$. If $N(w)\cap D=\set{v}$, we say that $w$ is a \emph{private neighbour} of $v$. The set of all private neighbours of a vertex $v\in D$ is called the \emph{private neighbourhood} of $v$. For any two vertices $v,w\in D$ which are at distance at most two, we say that $v$ \emph{witnesses} $w$ or that $v$ is a \emph{witness} of $w$. This terminology allows us to characterise a semitotal dominating set as a dominating set in which every vertex is witnessed by another vertex in the dominating~set.

We denote by $K_n$, $P_n$ and $C_n$ the complete graph, the path, and the cycle on $n$ vertices, respectively. We may also call $K_3$ a \emph{triangle}. For a path $P$ with endpoints $x$ and $y$, we call any vertex in $V(P)\setminus\set{x,y}$ an \emph{internal vertex} of $P$. The \emph{claw} is the complete bipartite graph with partition sizes one and three. For a graph $H$, we say that a graph $G$ is $H$-free if it does not contain $H$ as an induced subgraph. For a family of graphs $\mathcal{H}$, we say that a graph $G$ is $\mathcal{H}$-free if $G$ is $H$-free for every $H\in\mathcal{H}$. A graph is called \emph{chordal} if it is $C_k$-free for every $k\geq 4$.


\section{Structural results}
\label{sec:struct}

In this section, we present some structural results which will then be used in Section~\ref{sec:probl}. These results are comparable to those obtained by Huang and Xu \cite{HX10} for the domination and the total domination numbers. Observe that by definition, $\gamma_{t2}(G) \geq 2$ for any graph $G$, which justifies the lower bound on the semitotal domination number in the following. 

\begin{theorem}
\label{thm:ct3}
For any graph $G$ with $\gamma_{t2}(G) \geq 3$, $ct_{\gamma_{t2}} (G) \leq 3$.
\end{theorem}

\begin{proof}
Let $G$ be a graph with $\gamma_{t2}(G) \geq 3$ and let $D$ be a minimum semitotal dominating set of $G$. Consider $u,v \in D$ such that $d_G(u,v) \leq 2$ and let $w \in D \setminus \{u,v\}$ be a closest vertex to $\{u,v\}$, that is, $d_G(w,\{u,v\}) = d_G(D \setminus\{u,v\},\{u,v\}) = \min_{x \in D \setminus \{u,v\}} d_G(x, \{u,v\})$. We claim that $d_G(w,\{u,v\}) \leq 3$. Indeed, if $d_G(w,\{u,v\}) > 3$ and $x$ is the vertex at distance two from $w$ on a shortest path from $w$ to $\{u,v\}$, then $x$ is nonadjacent to $w,u$ and $v$ and so there exists a vertex $y \in D \setminus \{w,u,v\}$ adjacent to $x$ for otherwise $x$ would not be dominated. But then $d_G(y,\{u,v\}) < d_G(w,\{u,v\})$, a contradiction to the choice of $w$. Thus $d_G(w,\{u,v\}) \leq 3$. Now assume without loss of generality that $d_G(w,\{u,v\}) =  d_G(w,u)$ and let $P$ be a shortest path from $w$ to $u$. We claim that the graph $G'$ obtained by contracting the edges of $P$ has a semitotal domination number strictly less than that of $G$. Indeed, denote by $v_P$ the vertex resulting from the contraction of the edges of $P$ and let $D' = (D \setminus \{u,w\}) \cup \{v_P\}$. Then $D'$ is a semitotal dominating set of $G'$ as every vertex $x \in V(G) \setminus V(P)$ adjacent to a vertex of $P$ in $G$ is adjacent to $v_P$ in $G'$, and $d_{G'}(v_P,v) = d_G(u,v) \leq 2$. Thus $\gamma_{t2}(G') \leq |D'| = \gamma_{t2}(G) -1$ and since $P$ has length at most three, this concludes the proof.
\end{proof}

Next, we give necessary and sufficient conditions for $ct_{\gamma_{t2}}$ to be equal to one or two. Given a graph $G$, a \textit{friendly triple} is a subset of three vertices $x$, $y$ and $z$ such that $xy \in E(G)$ and $d_G(y,z) \leq 2$. The \emph{ST-configurations} correspond to the set of configurations depicted in \Cref{fig:conf}.

\begin{theorem}\label{thm:friendlytriple}
For any graph $G$, the following holds.
\begin{itemize}
\item[(i)] $ct_{\gamma_{t2}}(G) = 1$ if and only if there exists a minimum semitotal dominating set $D$ of $G$ such that $D$ contains a friendly triple.
\item[(ii)] $ct_{\gamma_{t2}}(G) = 2$ if and only if no minimum semitotal dominating set of $G$ contains a friendly triple and there exists a semitotal dominating set of size $ \gamma_{t2}(G) + 1$ that contains an ST-configuration. 
\end{itemize}
\end{theorem}

\newcommand\oone{$O_1$}
\newcommand\otwo{$O_2$}
\newcommand\othree{$O_3$}
\newcommand\ofour{$O_4$}
\newcommand\osix{$O_5$}
\newcommand\oseven{$O_6$}
\newcommand\oeight{$O_7$}

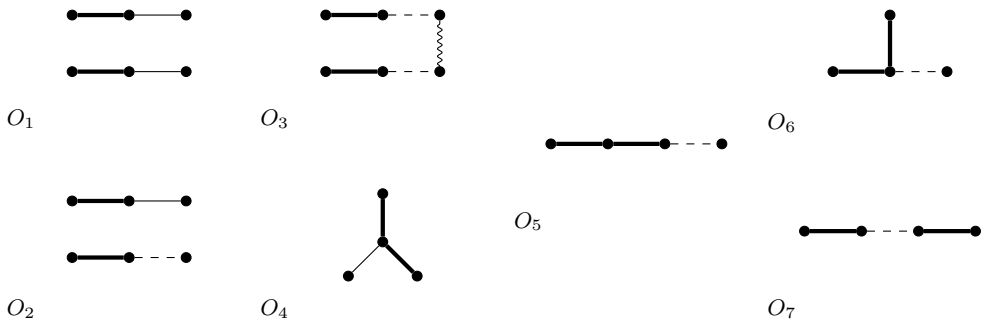
\begin{figure}[htb]
\begin{minipage}[b]{0.23\textwidth}
\begin{subfigure}[b]{\linewidth}
\centering
\begin{tikzpicture}[node distance=0.75cm]
\node[circ] (1) at (0,0) {};
\node[circ,right of=1] (2) {};
\node[circ,right of=2] (3) {};
\node[circ,above of=1] (4) {};
\node[circ,right of=4] (5) {};
\node[circ,right of=5] (6) {};
\node[invisible] at (1,-.25) {};

\draw[-] (2) -- (3)
(5) -- (6);
\draw[ultra thick] (1) -- (2)
(4) -- (5);
\end{tikzpicture}
\caption*{\oone}
\end{subfigure}

\vspace*{7mm}

\begin{subfigure}[b]{\linewidth}
\centering
\begin{tikzpicture}[node distance=0.75cm]
\node[circ] (1) at (0,0) {};
\node[circ,right of=1] (2) {};
\node[circ,right of=2] (3) {};
\node[circ,above of=1] (4) {};
\node[circ,right of=4] (5) {};
\node[circ,right of=5] (6) {};
\node[invisible] at (1,-.3) {}; 
\node[invisible] at (1,.9) {};

\draw[-] (5) -- (6);
\draw[dashed] (2) -- (3);
\draw[ultra thick] (1) -- (2)
(4) -- (5) ;
\end{tikzpicture}
\caption*{\otwo}
\end{subfigure}
\end{minipage}
\begin{minipage}[b]{0.23\textwidth}
\begin{subfigure}[b]{\linewidth}
\centering
\begin{tikzpicture}[node distance=0.75cm]
\node[circ] (1) at (0,0) {};
\node[circ,right of=1] (2) {};
\node[circ,right of=2] (3) {};
\node[circ,above of=1] (4) {};
\node[circ,right of=4] (5) {};
\node[circ,right of=5] (6) {};
\node[invisible] at (1,-.25) {};

\draw[ultra thick] (1) -- (2)
(4) -- (5);
\draw[dashed] (2) -- (3);
\draw[dashed] (5) -- (6);
\draw (3) edge[nonedge] (6);
\end{tikzpicture}
\caption*{\othree}
\end{subfigure}

\vspace*{7mm}

\begin{subfigure}[b]{\linewidth}
\centering
\begin{tikzpicture}[node distance=0.64cm]
\node[circ] (1) at (0,0) {};
\node[circ,above of=1] (2) {};
\node[circ,below right of =1] (3) {};
\node[circ,below left of=1] (4) {};

\draw[ultra thick] (1) -- (2)
(1) -- (3);
\draw[-](1) -- (4);
\end{tikzpicture}
\caption*{\ofour}
\end{subfigure}
\end{minipage}
\begin{minipage}[b]{0.23\textwidth}
\begin{subfigure}[b]{\linewidth}
\centering
\begin{tikzpicture}[node distance=0.75cm]
\node[circ] (1) at (0,0) {};
\node[circ,right of=1] (2) {};
\node[circ,right of=2] (3) {};
\node[circ,right of=3] (4) {};
\node[invisible] at (1,-.65) {};
\node[invisible] at (1,.5) {};

\draw[ultra thick] (1) -- (2)
(2) -- (3);
\draw[dashed] (3) -- (4);
\end{tikzpicture}
\caption*{\osix}
\end{subfigure}

\vspace*{7mm}

\begin{subfigure}[b]{\linewidth}

\end{subfigure}
\end{minipage}
\begin{minipage}[b]{0.23\textwidth}
\begin{subfigure}[b]{\linewidth}
\centering
\begin{tikzpicture}[node distance=0.75cm]
\node[circ] (1) at (0,0) {};
\node[circ,right of=1] (2) {};
\node[circ,right of=2] (3) {};
\node[circ,above of=2] (4) {};
\node[invisible] at (1,-.3) {};

\draw[ultra thick] (1) -- (2)
(2) -- (4);
\draw[dashed] (2) -- (3);
\end{tikzpicture}
\caption*{\oseven}
\end{subfigure}

\vspace*{7mm}

\begin{subfigure}[b]{\linewidth}
\centering
\begin{tikzpicture}[node distance=0.75cm]
\node[circ] (1) at (0,0) {};
\node[circ,right of=1] (2) {};
\node[circ,right of=2] (3) {};
\node[circ,right of=3] (4) {};
\node[invisible] at (1,-.65) {};
\node[invisible] at (1,.5) {};

\draw[ultra thick] (1) -- (2)
(3) -- (4);
\draw[dashed] (2) -- (3);
\end{tikzpicture}
\caption*{\oeight}
\end{subfigure}
\end{minipage}
\caption{The ST-configurations (the dashed lines indicate that the corresponding vertices are at distance 2 and the serpentine line indicates that the corresponding vertices can be identified). The thick edges correspond to the edges to contract in the proof of Theorem~\ref{thm:friendlytriple}(ii).} 
\label{fig:conf}
\end{figure}

\begin{proof}
Let $G$ be a graph. To prove $(i)$, let $D$ be a minimum semitotal dominating set of $G$ containing a friendly triple, that is, there is a subset of three vertices $x,y,z \in D$ such that $xy \in E(G)$ and $d_G(y,z) \leq 2$. Let $G'$ be the graph obtained from $G$ by the contraction of the edge $xy$, and let $v_{xy}$ be the vertex resulting from this contraction (note that $d_{G'}(z,v_{xy})\leq 2$). Then it is easy to see that $\left(D\setminus \{x,y\}\right)\cup\{v_{xy}\}$ is a semitotal dominating set of $G'$ of size $\gamma_{t2}(G)-1$. Conversely, assume that $G$ has an edge $xy$ whose contraction decreases the semitotal domination number of $G$. Let $G'$ and $v_{xy}$ be the graph and the vertex obtained from this contraction, respectively. Let $D'$ be a minimum semitotal dominating set of $G'$ (note that $|D'|\leq \gamma_{t2}(G)-1$). If $v_{xy}\in D'$ then there exists $z\in D'$ such that $d_{G'}(z,v_{xy})\leq 2$; in particular, at least one vertex of $\{x,y\}$ is at distance at most two from~$z$ in $G$. It follows that $D=(D'\setminus\{v_{xy}\})\cup\{x,y\}$ is a semitotal dominating set of $G$ containing a friendly triple, namely $x,y$ and $z$. Moreover, $D$ is minimum since $|D'|\leq \gamma_{t2}(G)-1$ and $|D|=|D'|+1$. Now assume that $v_{xy}\notin D'$. Since $v_{xy}$ is dominated in $D'$, at least one vertex of $\{x,y\}$ is dominated by a vertex of $D'$ in $G$, say $x$ without loss of generality. Consider the set $D=D'\cup\{x\}$ in $G$ (note that since $|D'|\leq \gamma_{t2}(G)-1$, $|D|\leq \gamma_{t2}(G)$). We will show that $D$ is a semitotal dominating set of $G$. It is easy to see that $D$ dominates every vertex of $G$ and $|D| = \gamma_{t2}(G)$. It remains to show that every vertex of $D$ has a witness. This holds for $x$: a witness for $x$ is any vertex $z\in D'$ (thus $z\in D$) that dominates $v_{xy}$ in $G'$ and is adjacent to $x$ in $G$ (such a vertex exists by the assumption that $x$ is dominated in $D'$). Now consider a vertex $p\in D\setminus\{x\}$ (note that $p\in D'$) and let $p'$ be a witness for $p$ in $D'$. If $pp'\in E(G)$ or if there exists a path $pup'$ in $G'$ with $u\neq v_{xy}$, then $p'$ is still a witness for $p$ in $D$. If a path of length at most two between $p$ and $p'$ in $G'$ contains $v_{xy}$ as an internal vertex, then $d_G(x,p)\leq 2$ and thus $x$ is a witness for $p$ in $D$. Hence every vertex in $D\setminus\{x\}$ has a witness and thus, $D$ is a semitotal dominating set of $G$. Finally observe that $D$ contains a friendly triple: indeed, denoting by $w$ a witness for $z$ in $D'$, we have that $w\in D$ and since $xz\in E(G)$, we conclude that $\{x,z,w\}$ is a friendly triple in~$D$ which completes the proof of~$(i)$.

We now proceed to the proof of~$(ii)$. If no minimum semitotal dominating set of $G$ contains a friendly triple then by $(i)$, $ct_{\gamma_{t2}}(G) >1$. Suppose that $G$ has a semitotal dominating set $S$ of size $\gamma_{t2}(G)+1$ such that $S$ contains an ST-configuration. It is straightforward to see that, for each configuration, the contraction of the two thick edges in Figure~\ref{fig:conf} reduces the size of $S$ by two. Moreover, after these contractions, $S$ remains a semitotal dominating set of the resulting graph. Thus we conclude that the contraction of two edges reduces the semitotal domination number of $G$ and hence $ct_{\gamma_{t2}}(G)=2$. 

For the other direction, let $e$ and $e'$ be two edges whose contraction decreases the semitotal domination number of $G$. In the remainder of this proof, we denote by $G'$ the graph obtained from $G$ by the contraction of the edges $e$ and $e'$ and by $D'$ a minimum semitotal dominating set of $G'$. Note that $|D'|=\gamma_{t2}(G)-1$ as $ct_{\gamma_{t2}}(G)>1$ and the contraction of a single edge decreases the semitotal domination number of a graph by at most one. We start with the following observation that will be useful throughout the proof.

\begin{observation}\label{obs:DhasP4}
Let $D$ be a semitotal dominating set of $G$. If $D$ contains a (not necessarily induced) $P_4$, then $D$ contains Configuration~\ofour~or Configuration~\oseven.
\end{observation}

Indeed, let $D$ be a semitotal dominating set of $G$ containing a (not necessarily induced) $P_4$ on vertex set $\{a,b,c,d\}$ with $ab,bc,cd\in E(G)$. If $ac\in E(G)$ then $\{a,b,c,d\}$ contains \ofour~in $D$ since $ac,bc,cd\in E(G)$. Otherwise $ac\notin E(G)$ in which case $d_G(a,c)=2$ as $b$ is a common neighbour of $a$ and $c$, and thus $\{a,b,c,d\}$ forms an \oseven~in $D$ as $bc,cd\in E(G)$.\\

We now consider the following cases.\\

\noindent{\bf Case 1.} \emph{$e$ and $e'$ share a vertex.} Let $e=xy$ and $e'=yz$ and let $v_{xyz}$ be the vertex of $G'$ resulting from the contraction of $e$ and $e'$. 

{\bf Case 1.1.} \textit{$v_{xyz}\notin D'$}. First note that, in this case, $D=D'\cup\{x,y\}$ is a semitotal dominating set of $G$ (of size $\gamma_{t2}(G)+1$). Indeed, $D$ is a dominating set since $D'$ is a dominating set of $G'$ and $y$ dominates $z$. Moreover, $x$ is a witness for $y$ (and vice versa) and if there is a vertex $p$ with witness $p'$ in $D'$ such that the unique path of length two connecting $p$ to $p'$ in $G'$ contained $v_{xyz}$, we have that $d_G(p,y)\leq 2$ and therefore $y$ is now a witness for $p$. Using similar arguments, we can show that $D'\cup\{y,z\}$ is also a semitotal dominating set of $G$.

Now since $D'$ is a dominating set of $G'$, at least one vertex of $\{x,y,z\}$ is dominated by $D'$ in $G$. Suppose first that $D'$ dominates $x$ in $G$ and consider the set $D=D'\cup\{x,y\}$. We next show that $D$ contains an ST-configuration. Let $w_1\in D'$ be a vertex that dominates $x$ and let $w_1'$ be a witness for $w_1$ in $D'$. If $d_G(w_1,w_1')=2$ then $\{x,y,w_1,w_1'\}$ forms an \osix~in $D$; otherwise, $d_G(w_1,w_1')=1$ in which case $D$ contains a $P_4$ on vertex set $\{x,y,w_1,w_1'\}$ and so by Observation~\ref{obs:DhasP4}, $D$ contains an \ofour~or an \oseven. We conclude similarly in the case where $D'$ dominates $y$ (respectively\ $z$) by considering the semitotal dominating set $D=D'\cup\{x,y\}$ (respectively\ $D=D'\cup\{y,z\}$).

{\bf Case 1.2.} \textit{$v_{xyz}\in D'$}. We first show that $D=(D'\setminus\{v_{xyz}\})\cup\{x,y,z\}$ is a semitotal dominating set of $G$ (note that $|D| = \gamma_{t2}(G)+1$). It is easy to see that $D$ is a dominating set. Furthermore, if $v_{xyz}$ was a witness for a vertex $p$ in $D'$ then in $G$, $p$ is at distance at most two to a vertex of $\{x,y,z\}$ and thus $p$ has a witness in $D$.

We next show that $D$ contains an ST-configuration. Let $w\in D'$ be a witness for $v_{xyz}$ in $D'$. Suppose first that $d_{G'}(w,v_{xyz})=1$. If $wy\in E(G)$ then $\{x,y,z,w\}$ forms an \ofour~in $D$; otherwise, $w$ is adjacent to $x$ or $z$, in which case $D$ contains a~$P_4$ on vertex set $\{x,y,z,w\}$ and so by Observation~\ref{obs:DhasP4}, $D$ contains an \ofour~or an \oseven. Now if $d_{G'}(w,v_{xyz})=2$ then $wx,wy,wz\notin E(G)$ and $w$ is at distance two to a vertex of $\{x,y,z\}$. Then either $d_G(w,y)=2$ in which case $\{x,y,z,w\}$ forms an \oseven~in $D$; otherwise, the same set forms an \osix~in $D$.\\

\noindent{\bf Case 2.} \emph{$e$ and $e'$ do not share a vertex.} Let $e=xy$ and $e'=zw$ and let $v_{xy}$ and $v_{zw}$ be the vertices of $G'$ resulting from the contraction of $e$ and $e'$, respectively.

{\bf Case 2.1.} \textit{$D'\cap \{v_{xy}, v_{zw}\}=\emptyset$}. Since $D'$ dominates $v_{xy}$ and $v_{zw}$, at least one of $\{x,y\}$ is dominated by $D'$ and the same holds for $\{z,w\}$. Assume without loss of generality that $x$ and $z$ are dominated by $D'$ and let $D=D'\cup\{x,z\}$. Note that $D$ is a semitotal dominating set of $G$ of size $\gamma_{t2}(G)+1$. We next show that $D$ contains an ST-configuration. Let $w_1$ (respectively\ $w_2$) be a vertex of $D$ that dominates $x$ (respectively\ $z$). If $w_1=w_2$, let $w'$ be a witness of $w_1$ in $D'$. Then $\{x,z,w_1,w'\}$ forms an \ofour (if $d_G(w',w_1)=1$) or an \oseven (if $d_G(w',w_1)=2$) in $D$. Suppose next that $w_1\neq w_2$ and let $w_1'$ (respectively\ $w_2'$) be a witness for $w_1$ (respectively\ $w_2$) in $D'$. Assume first that $w_1'=w_2'$. If $d_G(w_2,w_1')=1$ and $d_G(w_1,w_1')=1$, then $D$ contains a $P_4$ on vertex set $\{x,w_1,w'_1,w_2\}$ and so by Observation~\ref{obs:DhasP4}, $D$ contains an \ofour~or an \oseven. If $d_G(w_2,w'_1)=1$ and $d_G(w_1,w'_1)=2$ then $\{w_1,w'_1,w_2,z\}$ forms an \osix~in $D$. Finally, if both $d_G(w_2,w'_1)=2$ and $d_G(w_1,w'_1)=2$ then $\{x, w_1,w'_1,w_2,z\}$ forms an \othree~in $D$. Assume henceforth that $w_1'\neq w_2'$. If $w_1'=w_2$ and $w_1w_2\in E(G)$, then $D$ contains a $P_4$ on vertex set $\{x,w_1,w_2,z\}$ and so by Observation~\ref{obs:DhasP4}, $D$ contains an \ofour~or an \oseven. If $w_1'=w_2$ and $w_1w_2\notin E(G)$, then $\{x,w_1,w_2,z\}$ forms an \oeight~in $D$. Finally, assume that $w_1,w_1',w_2,w_2'$ are four distinct vertices in $G$. If $d_G(w_1,w_1')=1$ and $d_G(w_2,w_2')=1$, then $\{x,z,w_1,w_1',w_2,w_2'\}$ forms an \oone~in $D$. If $d_G(w_1,w_1')=1$ and $d_G(w_2,w_2')=2$, then $\{x,z,w_1,w_1',w_2,w_2'\}$ forms an \otwo~in $D$. Finally, if both $d_G(w_1,w_1')=2$ and $d_G(w_2,w_2')=2$, then $\{x,z,w_1,w_1',w_2,w_2'\}$ forms an \othree~in $D$.

{\bf Case 2.2.} \textit{$D'\cap \{v_{xy}, v_{zw}\}\neq\emptyset$}. If $|D'\cap \{v_{xy}, v_{yz}\}|=1$ then assume, without loss of generality, that $v_{xy}\in D'$. Since $v_{zw}\notin D'$, there exists $z'\in D'$ such that $z'$ is adjacent to $z$ or $w$, say $zz'\in E(G)$ without loss of generality. Consider the set $D=(D'\setminus \{v_{xy}\})\cup\{x,y,z\}$ (note that $|D|=\gamma_{t2}(G)+1$) and let us show that $D$ contains an ST-configuration. Let $p$ be a witness of $v_{xy}$ in $D'$. Assume without loss of generality that $d_G(p,y)\leq 2$. Suppose first that $z'=v_{xy}$. If $d_G(p,y)=1$ and $zy\in E(G)$, then $\{x,y,z,p\}$ forms an \ofour~in $D$. If $d_G(p,y)=1$ and $zy\notin E(G)$, then $d_G(y,z)\leq 2$ and therefore $\{x,y,z,p\}$ forms an \oseven~in $D$. If $d_G(p,y)=2$ and $zy\in E(G)$, then $\{x,y,z,p\}$ forms an \oseven~in $D$. Finally, if $d_G(p,y)=2$ and $zx\in E(G)$, then $\{x,y,z,p\}$ forms an \osix~in $D$. Second, suppose that $z'\neq v_{xy}$. If $p=z'$, we have two possibilities: either $d_G(p,y)=1$ in which case $D$ contains a $P_4$ on vertex set $\{x,y,p,z\}$ and so by Observation~\ref{obs:DhasP4}, $D$ contains an \ofour~or an \oseven; or $d_G(p,y)=2$ in which case $\{x,y,p,z\}$ forms an \oeight~in $D$. Assume henceforth that $p\neq z'$ and let $z''$ be a witness of $z'$ in $D'$. If $z''=v_{xy}$, then either $y$ or $x$ is a witness of $z'$ in $D$; and by symmetry, we can assume that $d_G(y,z')\leq 2$. If $d_G(y,z')=1$ then by Observation~\ref{obs:DhasP4}, $\{z,z',x,y\}$ forms an \ofour~or an~\oseven~in $D$. If $d_G(y,z')=2$ then the same set forms an \oeight~ in $D$. Hence, we can safely assume that $z''\neq v_{xy}$. Now note that $\{z,z',z''\}$ and $\{x,y,p\}$ form friendly triples in $D$ (recall that $d_G(p,y)\leq 2$) and it may still be the case that $p=z''$. However in this case, if $d_G(z',z'')=1$ and $d_G(z'',y)=1$ then by Observation~\ref{obs:DhasP4}, we have either an \ofour~or an \oseven~in $D$; if $d_G(z',z'')=1$ and $d_G(z'',y)=2$ then $\{z,z',z'',y\}$ forms an \osix~in $D$; and if $d_G(z',z'')=2$ and $d_G(z'',y)=1$ (respectively\ $d_G(z'',y)=2$), then $\{z,z',z'',y\}$ (respectively\ $\{z,z',z'',y,x\}$) forms an \oeight~(respectively \othree) in $D$. Thus we may assume that $p\neq z''$. Then $\{z,z',z''\}$ and $\{x,y,p\}$ are two disjoint friendly triples in $D$ and thus $\{z,z',z'',x,y,p\}$ forms either an \oone~(if $d_G(y,p)=d_G(z',z'')=1$), an \otwo~(if exactly one of $d_G(y,p)$ or $d_G(z',z'')$ equals two) or an \othree~(if $d_G(y,p)=d_G(z',z'')=2$).

We conclude the proof by considering the case where $\{v_{xy},v_{zw}\}\subseteq D'$ to which a similar case analysis applies. Consider the set $D=(D'\setminus\{ v_{xy},v_{zw}\})\cup\{x,y,z,w\}$ (note that $D$ is a semitotal dominating set of $G$ of size $\gamma_{t2}(G)+1$) and let us show that $D$ contains an ST-configuration. If a vertex of $\{x,y\}$ is adjacent to a vertex of $\{z,w\}$, then $D$ contains a $P_4$ and so by Observation~\ref{obs:DhasP4}, $\{x,y,z,w\}$ forms an \ofour~or an~\oseven~in $D$. If a vertex of $\{x,y\}$ is at distance exactly two from a vertex in $\{z,w\}$, then $D$ contains an \oeight. If neither of these conditions hold, that is, if $d_{G'}(v_{xy},v_{zw})\geq 3$, let $p$ (resp.\ $p'$) be a witness for $v_{xy}$ (resp.\ $v_{zw}$) in $D'$. Note that $p\neq v_{zw}$ and $p'\neq v_{xy}$ since $d_{G'}(v_{xy},v_{zw})\geq 3$. Hence, if $p=p'$ then $D$ contains an \othree~or an \osix; and if $p\neq p'$, then $\{x,y,p\}$ and $\{z,w,p'\}$ are two disjoint friendly triples in $D$ and thus, $\{x,y,z,w,p,p'\}$ forms either an \oone, an \otwo~or an \othree~in $D$, which concludes the proof.
\end{proof}


\section{The complexity of \contracstd{}}
\label{sec:probl}

In this section, we consider several graph classes and determine for each of them whether \contracstd{} is $\mathsf{(co)NP}$-hard (Section \ref{subsec:hard}) or polynomial-time solvable (Section \ref{subsec:poly}). Putting these results together then leads to our main theorem (Section~\ref{sec:main}).

\subsection{Hardness results}
\label{subsec:hard}

Similarly to the case of domination, we have the two following results.

\begin{theorem}[\app]
\label{thm:claw}
\contracstd{} is $\mathsf{coNP}$-hard when restricted to claw-free graphs.
\end{theorem}

\begin{theorem}[\app]
\label{thm:2P3}
\contracstd{} is $\mathsf{coNP}$-hard when restricted to $2P_3$-free graphs.
\end{theorem}

We next focus on $\mathcal{C}$-free graphs, where $\mathcal{C}$ is a (possibly infinite) family of cycles, and show a relation between \contracd{} and \contracstd{}.

\begin{lemma}
\label{lem:cycles}
Let $\mathcal{C}$ be a (possibly infinite) family of cycles. If \contracd{} is $\mathsf{NP}$-hard when restricted to $\mathcal{C}$-free graphs then \contracstd{} is $\mathsf{NP}$-hard when restricted to $\mathcal{C}$-free graphs.
\end{lemma}

\begin{proof}
Let $G$ be a $\mathcal{C}$-free graph. We construct a $\mathcal{C}$-free graph $T(G)$ such that $G$ is a \yes-instance for \contracd{} if and only if $T(G)$ is a \yes-instance for \contracstd{} as follows. For every vertex $v \in V(G)$, we attach a copy of the tree $T_v$ depicted in Figure~\ref{fig:Tv} by connecting $v$ to $a_v$. We let $T(G)$  be the resulting graph. Clearly, $T(G)$ is $\mathcal{C}$-free.

\begin{figure}[htb]
    \centering
\begin{tikzpicture}{node distance=1cm}
\node[circ,label=above:{\small $a_v$}] (av) at (0,0) {};
\node[circ,label=above:{\small $b_v$}] (bv) at (1.5,0) {};
\node[circ,below left of=bv,label=below:{\small $y_1^v$}] (y1v) {};
\node[circ,below of=bv,label=below:{\small $y_2^v$}] (y2v) {};
\node[circ,below right of=bv,label=below:{\small $y_3^v$}] (y3v) {};
\node[circ,label=above:{\small $c_v$}] at (3,0) {};
\node[circ,label=above:{\small $d_v$}] (dv) at (4.5,0) {};
\node[circ,below left of=dv,label=below:{\small $x_1^v$}] (x1v) {};
\node[circ,below of=dv,label=below:{\small $x_2^v$}] (x2v) {};
\node[circ,below right of=dv,label=below:{\small $x_3^v$}] (x3v) {};

\draw[-] (av) -- (dv)
(bv) -- (y1v)
(bv) -- (y2v)
(bv) -- (y3v)
(dv) -- (x1v)
(dv) -- (x2v)
(dv) -- (x3v);
\end{tikzpicture}
\caption{The tree $T_v$.}
\label{fig:Tv}
\end{figure}

Let us first show that $\gamma_{t2}(T(G)) = \gamma(G) + 2|V(G)|$. Clearly, if $D$ is a minimum dominating set of G then $D \cup \{b_v,d_v\colon\,v \in V(G)\}$ is a semitotal dominating set of $T(G)$. Thus, $\gamma_{t2}(T(G)) \leq \gamma(G) + 2|V(G)|$. Let $D$ be a minimum semitotal dominating set of $T(G)$. Now $b_v \in D$ for every $v \in V(G)$. Indeed, if $b_v$ were not in $D$ then $y_1^v,y_2^v,y_3^v \in D$, since $y_i^v$ must be dominated for every $i\in [3]$. But then $(D\setminus\{y_1^v,y_2^v\})\cup \{b_v\}$ is a semitotal dominating set of $T(G)$ of size strictly less than $|D|$, a contradiction to the minimality of $D$. Using similar arguments, we can show that $d_v \in D$ for every $v \in V(G)$. This implies that $c_v\notin D$ for every $v \in V(G)$. Further, if $a_v \in D$ for some $v\in V(G)$ then $(D \setminus \{a_v\}) \cup \{v\}$ is a semitotal dominating set of $T(G)$ of size at most $|D|$. Thus, $T(G)$ has a minimum semitotal dominating set $D$ such that $D \cap \{a_v\colon\, v\in V(G)\} = \emptyset$ and since $b_v,d_v \in D$ for every $v \in V(G)$, in fact $D \cap V(T_v) = \{b_v,d_v\}$ for every $v \in V(G)$. From now on, we assume that $D$ is such a minimum semitotal dominating set and claim that $D \setminus \{b_v,d_v, v\in V(G)\}$ is a dominating set of $G$. Indeed, since for every $v \in V(G)$, $a_v \notin D$ necessarily $D \cap (N_{T(G)}[v] \setminus \{a_v\}) \neq \emptyset$ for otherwise $v$ would not be dominated in $D$. Thus $\gamma(G) \leq \gamma_{t2}(T(G)) - 2|V(G)|$ and combined with the above inequality, we conclude that in fact equality holds.

Now assume that $G$ is a \yes-instance for \contracd{} and let $D$ be a minimum dominating set of $G$ containing at least one edge $xy \in E(G)$ (see Theorem~\ref{theorem:contracdom}(i)). Then clearly $D \cup \{b_v,d_v\colon\, v \in V(G)\}$ is a minimum semitotal dominating set containing a friendly triple, namely $x,y,b_y$.

Conversely, assume that $T(G)$ is a \yes-instance for \contracstd{} and let $D$ be a minimum semitotal dominating set containing a friendly triple (see Theorem~\ref{thm:friendlytriple}), say $x,y,z$ where $x$ and $y$ are adjacent and $d_{T(G)}(y,z) \leq 2$. Now observe that either both 
$x$ and $y$ belong to $V(G)$, or there exists $v \in V(G)$ such that both $x$ and $y$ belong to $V(T_v)$. Indeed, if $x \in V(G)$ and $y \in V(T_v)$ for some $v \in V(G)$, then necessarily $v = x$ and $y = a_v$. But then since $b_v \in D$ by the above, $D \setminus \{a_v\}$ is a semitotal dominating set of $T(G)$ of size strictly less that $|D|$, a contradiction to the minimality of $D$. Now if both $x$ and $y$ belong to $V(G)$ then by the above, $(D \cap V(G)) \cup \{v: a_v \in D\}$ is a minimum dominating set of $G$ containing an edge, namely $xy$. Next, assume that there exists $v \in V(G)$ such that $x,y \in V(T_v)$. As shown above, $\{x,y\} \cap \{y_1^v,y_2^v,y_3^v,x_1^v,x_2^v,x_3^v,c_v\} = \emptyset$ (it would otherwise contradict the minimality of $D$ as $b_v,d_v \in D$) and so necessarily $\{x,y\} = \{a_v,b_v\}$. But then $v \notin D$ for otherwise $D \setminus \{a_v\}$ would be a semitotal dominating set of $T(G)$ of size strictly less than $|D|$. Now consider a neighbour $w \in V(G)$ of $v$. Then $w \notin D$ for otherwise $D \setminus \{a_v\}$ would be a semitotal dominating set of $T(G)$ of size strictly less than $|D|$, a contradiction to the minimality of $D$. But as $w$ is dominated in $D$, $w$ has a neighbour $u$ in $D$. If $u = a_w$ then by the above $(D \cap V(G)) \cup \{t: a_t \in D\}$ is a minimum dominating set of $G$ containing an edge, namely $wv$. Otherwise $u \in V(G)$ and so $(D \setminus \{a_v\}) \cup \{w\}$ is a minimum semitotal dominating set of $T(G)$ containing a friendly triple whose edge lies in $V(G)$, namely $u,w,b_w$, and we proceed as previously. Since in any case we can construct a minimum dominating set of $G$ containing an edge, we conclude by Theorem~\ref{theorem:contracdom}(i) that $G$ is a \yes-instance for \contracd{}, which completes the proof.  
\end{proof}

In \cite{domcontract}, the authors showed the following result for \contracd{}.

\begin{theorem}[\cite{domcontract}]
\label{theorem:largegirth}
\contracd{} is $\mathsf{NP}$-hard when restricted to $\{C_3,\ldots,C_\ell\}$-free graphs for any $\ell \geq 3$, and when restricted to bipartite graphs.
\end{theorem}

By combining Lemma~\ref{lem:cycles} and Theorem~\ref{theorem:largegirth}, we obtain the following.

\begin{theorem}
\contracstd{} is $\mathsf{NP}$-hard when restricted to $\{C_3,\ldots,C_\ell\}$-free graphs for any $\ell \geq 3$, and when restricted to bipartite graphs.
\end{theorem}

Finally, as for the case of domination, we can show the following. 

\begin{theorem}[\app]
\label{thm:P6P4+P2}
\contracstd{} is $\mathsf{NP}$-hard when restricted to $\{P_6,P_4+P_2\}$-free chordal graphs.
\end{theorem}


\subsection{Polynomial cases}
\label{subsec:poly}

We now focus on graph classes for which \contracstd{} can be solved in polynomial time. We start with some easy cases.

\begin{proposition}
\label{prop:boundedstdom}
\contracstd{} can be solved in polynomial time for a graph class $\mathcal{C}$, if either
\begin{itemize}
\item[(a)] $\mathcal{C}$ is closed under edge contractions and \textsc{Semitotal Dominating Set} can be solved in polynomial time on $\mathcal{C}$; or
\item[(b)] for every $G\in\mathcal{C}$, $\gamma_{t2}(G)\leq q$, where $q$ is some fixed constant; or 
\item[(c)] $\mathcal{C}$ is the class of $(H+K_1)$-free graphs, where $|V(H)|=q$ is a fixed constant and \contracstd is polynomial-time solvable on $H$-free graphs.
\end{itemize}
\end{proposition}

\begin{proof}
In order to prove item (a), it suffices to note that if we can compute $\gamma_{t2}(G)$ and $\gamma_{t2}(G\setminus e)$, for any edge $e$ of $G$, in polynomial time, then we can determine whether a graph $G$ is a \yes-instance for \contracstd{} in polynomial time.

For item (b), we proceed as follows. Given a graph $G$ of $\mathcal{C}$, we consider every subset $S \subseteq V(G)$ with $\vert S \vert \leq q$ and check whether it is a semitotal dominating set of $G$. Since there are at most $O(n^q)$ possible such subsets, we can determine the semitotal domination number of $G$ and check whether the conditions given in Theorem~\ref{thm:friendlytriple}(i) are satisfied in polynomial time.

Finally, so as to prove item (c), we provide the following algorithm. Let $H$ and $q$ be as stated and let $G$ be an instance of \contracstd{} on $(H+K_1)$-free graphs. We first test whether $G$ is $H$-free (note that this can be done in time $O(n^q)$). If this is the case, we use the polynomial-time algorithm for \contracstd{} on $H$-free graphs. Otherwise, $G$ has an induced subgraph isomorphic to $H$; but since $G$ is a $(H+K_1)$-free graph, $V(H)$ must then be a dominating set of $G$ and so, $\gamma_{t2}(G)\leq 2q$. We then conclude by Proposition~\ref{prop:boundedstdom}(b) that \contracstd{} is also polynomial-time solvable in this case.
\end{proof}

In what follows, we will use the following result by Galby et al.~\cite{domcontract}.

\begin{lemma}[\cite{domcontract}]
\label{lemma:p5free}
If $G$ is a $P_5$-free graph and $\gamma(G)\geq 3$, then $ct_{\gamma} (G) = 1$.
\end{lemma}

\begin{lemma}
\label{lem:P5free}
\contracstd{} is polynomial-time solvable on $P_5$-free graphs.
\end{lemma}

\begin{proof}
Let $G$ be a $P_5$-free graph. If $\gamma _{t2} (G) = 2$ then $G$ is clearly a \no-instance for \contracstd{}. Assume henceforth that $\gamma_{t2} (G) \geq 3$. Since $G$ is $P_5$-free, $G$ is in particular $(C_6,P_6,P)$-free (see Figure~\ref{fig:P}); it then follows from \cite{semitot} that $\gamma (G) = \gamma_{t2} (G)$. Now by Lemma \ref{lemma:p5free}, $ct_{\gamma} (G) = 1$ which implies that there exists a minimum dominating set of $G$ which is not independent (see Theorem~\ref{theorem:contracdom}(i)). Amongst those non-independent minimum dominating sets, consider one $D$ with the fewest unwitnessed vertices. Let us show that $D$ is a semitotal dominating set. 

\begin{figure}[t]
    \centering
\begin{tikzpicture}
\node[circ] (1) at (0,0) {};
\node[circ] (2) at (1,0) {};
\node[circ] (3) at (0,1) {};
\node[circ] (4) at (1,1) {};
\node[circ] at (2,0) {};
\node[circ] (6) at (3,0) {};

\draw[-] (1) -- (6)
(1) -- (3)
(2) -- (4)
(3) -- (4);
\end{tikzpicture}
    \caption{The graph $P$.}
    \label{fig:P}
\end{figure}

Suppose to the contrary that there exists $w \in D$ such that $w$ has no witness and let $u \in D$ be a vertex such that $d_G(w,D\backslash \{w\}) = d_G(w,u)$. Since $G$ is $P_5$-free, it follows that $d_G(u,w) \leq 3$, and as $d_G(u,w) > 2$ by assumption, in fact $d_G(u,w) = 3$. Let $x$ (respectively $y$) be the neighbour of $u$ (respectively $w$) on a shortest path from $u$ to $w$. We claim that $N_G(u) \cup N_G(w) \subseteq N_G(x) \cup N_G(y)$; indeed, if $a$ is a neighbour of $u$ then $a$ is nonadjacent to $w$ (otherwise $d_G(u,w) \leq 2$) and thus, $a$ is adjacent to either $x$ or $y$ for otherwise $auxyw$ would induce a $P_5$. We conclude similarly if $a$ is a neighbour of $w$. But then, $(D \backslash \{u,w\}) \cup \{x,y\})$ is a dominating set which is not independent and contains fewer unwitnessed vertices than $D$, a contradiction to its minimality. Thus, $D$ is a minimum semitotal dominating set. 

Now consider $u,v \in D$ such that $uv \in E(G)$. If there exists $w \in D$ such that $d_G(w, \{u,v\}) \leq 2$, then $u,v,w$ is a friendly triple contained in $D$ and we conclude by Theorem~\ref{thm:friendlytriple}(i). Assume henceforth that no such vertex exists and consider a vertex $w \in D$ the closest to $\{u,v\}$. Since $G$ is $P_5$-free, it follows that $d_G(w, \{u,v\}) \leq 3$, and as $d_G(w,\{u,v\}) > 2$ by assumption, in fact $d_G(w, \{u,v\}) = 3$. Assume, without loss of generality, that $d_G(w,v) \geq d_G(w,u) =3$ and denote by $x$ (respectively $y$) the neighbour of $u$ (respectively $w$) on a shortest path from $u$ to $w$. Then, as previously, we have that $N_G(w) \cup N_G(u) \subseteq N_G(x) \cup N_G(y)$ and thus, $D'=(D\backslash \{u,w\}) \cup \{x,y\}$ is a minimum semitotal dominating set containing a friendly triple, namely $x,y,v$ (note that by assumption, no vertex in $D$ had $u$ or $v$ as a witness and so $D'$ is indeed a semitotal dominating set). Hence by Theorem~\ref{thm:friendlytriple}(i), $ct_{\gamma_{t2}}(G) = 1$ which concludes the proof.
\end{proof}

By combining Lemma~\ref{lem:P5free} and Proposition~\ref{prop:boundedstdom}(c), we obtain the following.

\begin{theorem}
\label{thm:P4}
For any $t \geq 0$, \contracstd{} is polynomial-time solvable on $(P_5 + tK_1)$-free graphs.
\end{theorem}

Let us now present the last result of this section regarding $P_3+kP_2$-free graphs.

\begin{theorem}
\label{thm:P3kP2}
For any $k \geq 0$, \contracstd{} is polynomial-time solvable on $P_3+kP_2$-free graphs.
\end{theorem}

\begin{proof}
First observe that if $G$ does not contain an induced $P_3$ then $G$ is a disjoint union of cliques and thus a \no-instance for \contracstd. Assume henceforth that $k \geq 1$ and let $G$ be a $P_3+kP_2$-free graph containing an induced $P_3+(k-1)P_2$. The following proof is similar to that of \cite[Theorem 2]{P3kP2}. Let $A \subseteq V(G)$ be a set of vertices which induces a $P_3+(k-1)P_2$, let $B \subset V(G)$ be the set of vertices at distance one from $A$ and let $C \subset V(G)$ be the set of vertices at distance two from $A$. Note that since $G$ is $P_3+kP_2$-free, the sets $A,B$ and $C$ partition $V(G)$ and $C$ is an independent set. We call a vertex $v_1 \in C$ a \emph{regular vertex} if there exist $k$ vertices $v_2,\ldots,v_{k+1} \in C$ such that $v_1,\ldots,v_{k+1}$ are pairwise at distance at least four and $N(v_i)$ is a clique for every $i \in [k+1]$. We denote by $\mathcal{R}$ the set of regular vertices. 

\begin{claim}
\label{clm:Rnonempty}
If $\mathcal{R} \neq \varnothing$ then the following holds.
\begin{itemize}
\item[(i)] $\gamma(G) = \gamma_{t2}(G)$.
\item[(ii)] $G$ is a \yes-instance for \contracd{} if and only if $G$ is a \yes-instance for \contracstd.
\end{itemize}
\end{claim}

\begin{claimproof}
Let $V_1\subseteq V(G)\setminus N[\mathcal{R}]$ be the set of vertices at distance one from $N[\mathcal{R}]$ and let $V_2=V(G)\setminus (N[\mathcal{R}]\cup V_1)$. Note that since $G[N[\mathcal{R}]]$ contains an induced $kP_2$, $V_2$ is $P_3$-free and thus $G[V_2]$ is a disjoint union of cliques. 

Let $D$ be a minimum dominating set of $G$. We show how to transform $D$ into a semitotal dominating set of $G$ of the same size. It is shown in \cite[Claim 7]{P3kP2} that $\vert D\cap N[c]\vert =1$ for every regular vertex $c$. Thus, if $D$ contains a regular vertex $c$ then $D \cap N(c) = \varnothing$ and the set $(D \setminus \{c\}) \cup \{b\}$ where $b \in N(c)$, is also a minimum dominating set of $G$. Furthermore, if $D$ contains an edge $e$ then $c$ is not an endpoint of $e$. Hence, we may replace every regular vertex in $D$ by one of its neighbour (without destroying any edge contained in $D$). Now suppose that $D$ contains a vertex $v \in V_2$ which is anticomplete to $V_1$ and denote by $C_v$ the clique of $V_2$ containing $v$. Since $G$ is connected, $C_v \setminus \{v\} \neq \varnothing$ and $N_v = \{u \in C_v \setminus \{v\}~|~N(u) \cap V_1 \neq \varnothing\} \neq \varnothing$. If $N_v \cap D = \varnothing$ then the set $(D \setminus \{v\}) \cup \{u\}$ where $u \in N_v$, is also a minimum dominating set of $G$; furthermore, if $G$ contains an edge $e$ then $v$ is not an endpoint of $e$. If $N_v \cap D \neq \varnothing$, then $D$ would clearly not be minimum. Hence, we may replace in $D$ every vertex of $D \cap V_2$ which is anticomplete to $V_1$ either by a vertex in $V_2$ adjacent to $V_1$ or by a vertex in $V_1$ (while preserving the property of containing an edge). We claim that then $D$ is a semitotal dominating set. Indeed, it is shown in \cite{P3kP2}[Claim 6] that if a vertex $b \in V(G) \setminus N[\mathcal{R}]$ is adjacent to $N(c)$ for some regular vertex $c \in \mathcal{R}$ then there exists $c' \in \mathcal{R} \setminus \{c\}$ such that $b$ is complete to $N(c')$. By applying this claim twice, it follows that for every vertex $v\in V_1$ there exist two regular vertices $c,c'\in\mathcal{R}$ such that $v$ is complete to $N(c) \cup N(c')$. Since $|D \cap N(c)| = 1$ for every $c \in \mathcal{R}$, this implies that every vertex in $D \cap V_1$ has a witness in $D \cap N(\mathcal{R})$. Furthermore since $G$ is connected, the above also implies that for any $c \in \mathcal{R}$, every vertex in $N(c)$ is within distance at most two from every vertex in $N(c')$ for some $c' \in \mathcal{R} \setminus \{c\}$; thus every vertex in $D \cap N(\mathcal{R})$ is witnessed by some vertex in $D \cap N(\mathcal{R})$. Now since any vertex $v \in D \cap V_2$ is adjacent to some vertex in $V_1$, $v$ is within distance two of a vertex in $N(\mathcal{R})$ (recall that every vertex in $V_1$ is complete to $N(c)$ for some regular vertex $c$) and thus, within distance two of a vertex in $D \cap N(\mathcal{R})$. It follows that $D$ is a semitotal dominating set of $G$ and since $\gamma(H) \leq \gamma_{t2}(H)$ for any graph $H$, we conclude that $\gamma(G) = \gamma_{t2}(G)$. 

Now suppose that $D$ initially contained an edge, that is, $G$ is a \yes-instance for \textsc{$1$-Edge Contrac\-tion($\gamma$)}. Then as shown above, the transformed $D$ also contains an edge $e=uv$. Suppose first that $u \in N(c)$ and $v \in N(c')$ for some $c,c' \in \mathcal{R}$ (note that $c \neq c'$ as $|D \cap N(v)| = 1$ for every regular vertex $v$). Since $c$ is a regular vertex, there exist $c_1,\ldots,c_k \in \mathcal{R} \setminus \{c'\}$ such that $c,c_1,\ldots,c_k$ are pairwise at distance at least four. For every $i \in [k]$, denote by $v_i$ the vertex in $D \cap N(c_i)$. Then there exists $j \in [k]$ such that $v$ is adjacent to $v_j$ for otherwise $v,u,c,c_1,v_1,\ldots,c_k,v_k$ induce a $P_3+kP_2$; thus $u,v,v_j$ is a friendly triple. Now if one of $u$ and $v$ belongs to $V_1$, say $u \in V_1$ without loss of generality, then by the above there exist $x,y \in D \cap N(\mathcal{R})$ such that $u$ is adjacent to both $x$ and $y$. Assuming without loss of generality that $v \neq y$, we then have that $u,v,y$ is a friendly triple. Finally, if $u,v \in V_2$ then $u$ is adjacent to some vertex $w \in V_1$ which by the above is adjacent to a vertex $c \in D \cap N(\mathcal{R})$ and so $u,v,c$ is a friendly triple. Since in every case we can find a friendly triple, we conclude by Theorem~\ref{thm:friendlytriple}(i) that $G$ is a \yes-instance for \contracstd. Conversely, if there exists a minimum semitotal dominating set $D$ of $G$ containing a friendly triple then $D$ is a fortiori a minimum dominating set of $G$ containing an edge; thus $G$ is a \yes-instance for \contracd{} which concludes the proof.
\end{claimproof}

\begin{claim}[\app]
\label{clm:Rempty}
If $\mathcal{R} = \varnothing$ and $G$ is a \no-instance for \contracstd{} then $\gamma_{t2}(G) \leq (k+1)(\vert A\vert+2)+k(1+2(k+1))+5\vert A\vert-4$.
\end{claim} 

\noindent
Consider now the following algorithm whose correctness is guaranteed by Claims~\ref{clm:Rnonempty} and \ref{clm:Rempty}.

\begin{itemize}
\item[1.] Compute $A$, $B$, $C$ and $\mathcal{R}$.
\item[2.] If $\mathcal{R} \neq \varnothing$ then check whether $G$ is a \yes-instance for \contracd{}. 
\begin{itemize}
\item[2.1] If the answer is yes then output \yes. 
\item[2.2] Otherwise output \no.
\end{itemize}
\item[3.] If $\mathcal{R} = \varnothing$ then check whether there exists a semitotal dominating set of size at most $k(|A|+2(k+1)+2)+3|A|+k-4$.
\begin{itemize}
\item[3.1] If the answer is no then output \yes.
\item[3.2] Otherwise, determine whether there exists a minimum dominating set containing friendly triple or not using brute force (see Theorem~\ref{thm:friendlytriple}(i)).
\end{itemize}
\end{itemize}

Regarding its complexity, it is shown in \cite{P3kP2} that checking whether $G$ is a \yes-instance for \contracd{} can be done in polynomial time; thus, step 2 can be done in polynomial-time. Now since step 1 clearly takes polynomial time and checking whether there exists a minimum semitotal dominating set of size at most $k(|A|+2(k+1)+2)+3|A|+k-4$ containing a friendly triple can also be done in polynomial time (by simple brute force), we conclude that the above algorithm runs in polynomial time.	
\end{proof}


\subsection{Proof of Theorem~\ref{thm:dichotomy3}}
\label{sec:main}

We finally prove Theorem~\ref{thm:dichotomy3}. Let $H$ be a graph. If $H$ contains a cycle then \contracstd{} is $\mathsf{NP}$-hard on $H$-free graphs by Theorem~\ref{theorem:largegirth}. Thus, we may assume that $H$ is a forest. If $H$ contains a vertex of degree at least three, then $H$ contains an induced claw and so \contracstd{} is $\mathsf{coNP}$-hard on $H$-free by Theorem~\ref{thm:claw}. Assume henceforth that $H$ is a linear forest. If $H$ contains a path on at least six vertices then \contracstd{} is $\mathsf{NP}$-hard on $H$-free graphs by Theorem~\ref{thm:P6P4+P2}. Thus we may assume that every connected component of $H$ induces a path on at most five vertices. Now suppose that $H$ contains a path on at least four vertices. If $H$ has another connected component on more than one vertex, then \contracstd{} is $\mathsf{NP}$-hard on $H$-free graphs by Theorem~\ref{thm:P6P4+P2}. Otherwise, every other connected component of $H$ (if any) contains exactly one vertex in which case \contracstd{} is polynomial-time solvable by Theorem~\ref{thm:P4}. Now suppose that the longest path in $H$ has length three. If $H$ has another connected component on three vertices then \contracstd{} is $\mathsf{coNP}$-hard by Theorem~\ref{thm:2P3}. Otherwise, every other connected component of $H$ (if any) has size at most two in which case \contracstd{} is polynomial-time solvable on $H$-free graphs by Theorem~\ref{thm:P3kP2}. Finally if every connected component of $H$ has size at most two then \contracstd{} is polynomial-time solvable on $H$-free graphs by Theorem~\ref{thm:P3kP2}, which concludes the proof.




\bibliography{arxiv}

\newpage

\appendix

\section{Proof of Theorem~\ref{thm:claw}}

In the following hardness proof, we reduce from the \textsc{Positive Exactly 3-Bounded 1-In-3 3-Sat} problem which is a variant of the \textsc{3-Sat} problem where given a formula $\Phi$ in which all literals are positive, every clause contains exactly three literals and every variable appears in exactly three clauses, the problem is to determine whether there exists a truth assignment such that each clause has exactly one true literal. This problem was shown to be $\mathsf{NP}$-complete in \cite{1IN3}.\\

We first introduce the following graph, called the \emph{long paw}, which we will use in the reduction.
\begin{figure}[htb]
\centering
\begin{tikzpicture}
\node[circ,label=below:{\small $P(1)$ }] (a) at (0,0) {};
\node[circ,label=below:{\small $P(2)$}] (b) at (1,0) {};
\node[circ,label=right:{\small $P(3)$}] (c) at (.5,.86) {};
\node[circ,label=right:{\small $P(4)$}] at (.5,1.36) {};
\node[circ,label=right:{\small $P(5)$}] (d) at (.5,1.86) {};

\draw (a) -- (b) 
(a) -- (c) 
(b) -- (c) 
(c) -- (d);
\end{tikzpicture}
\caption{The long paw $P$.}
\label{fig:LongPaw}
\end{figure}

As mentioned above, we reduce from \textsc{Positive Exactly 3-Bounded 1-In-3 3-Sat}: given an instance $\Phi$ of this problem, with variable set $X$ and clause set $C$, we construct an instance $G$ of \contracstd{} such that $\Phi$ is a \yes-instance for \textsc{Positive 1-In-3 3-Sat} if and only if $G$ is a \no-instance for \contracstd{}, as follows. For every variable $x \in X$ contained in clauses $c,c'$ and $c''$, we introduce the gadget $G_x$ depicted in Figure~\ref{fig:vargadclaw} (where the rectangles indicate that the corresponding set of vertices is a clique). For every clause $c \in C$ containing variables $x,y$ and $z$, we introduce the gadget $G_c$ depicted in Figure~\ref{fig:clausegadclaw} consisting of the disjoint union of the graph $G_c^T$ and the graph $G_c^F$. Finally, for every clause $c \in C$ containing variables $x,y$ and $z$, we add edges between the corresponding gadgets as follows. 
\begin{itemize}
\item[$\cdot$] For every $p \in \{x,y,z\}$,  we connect $P_{p,1}^c(2)$ to $f_c^{ab}$ if and only if $p \in \{a,b\}$.
\item[$\cdot$] For every $p \in \{x,y,z\}$, we connect $P_{p,2}^c(1)$ to $t_c^p$ and further connect $P_{p,2}^c(1)$ to $w_c^{ab}$ if and only if $p \in \{a,b\}$.   
\end{itemize}
We denote by $G$ the resulting graph.

\begin{figure}[t]
\centering
\begin{tikzpicture}
\node[circ,label=below:{\small $a_x^c$}] (ac) at (2,2) {};
\longPaw{3.5}{2}{$P_{x,1}^c$}{$P_{x,1}^c(1)$}{$P_{x,1}^c(2)$};
\draw (ac) -- (3.5,2);

\node[circ,label=below:{\small $a_x^{c'}$}] (ac') at (2,0) {};
\longPaw{3.5}{0}{$P_{x,1}^{c'}$}{$P_{x,1}^{c'}(1)$}{$P_{x,1}^{c'}(2)$};
\draw (ac') -- (3.5,0);

\node[circ,label=below:{\small $a_x^{c''}$}] (ac'') at (2,-2) {};
\longPaw{3.5}{-2}{$P_{x,1}^{c''}$}{$P_{x,1}^{c''}(1)$}{$P_{x,1}^{c''}(2)$};
\draw (ac'') -- (3.5,-2);

\node[circ,label=below:{\small $b_x^c$}] (dc) at (-2,2) {};
\longPaw{-4.5}{2}{$P_{x,2}^c$}{$P_{x,2}^c(1)$}{$P_{x,2}^c(2)$};
\draw (dc) -- (-3.5,2);

\node[circ,label=below:{\small $b_x^{c'}$}] (dc') at (-2,0) {};
\longPaw{-4.5}{0}{$P_{x,2}^{c'}$}{$P_{x,2}^{c'}(1)$}{$P_{x,2}^{c'}(2)$};
\draw (dc') -- (-3.5,0);

\node[circ,label=below:{\small $b_x^{c''}$}] (dc'') at (-2,-2) {};
\longPaw{-4.5}{-2}{$P_{x,2}^{c''}$}{$P_{x,2}^{c''}(1)$}{$P_{x,2}^{c''}(2)$};
\draw (dc'') -- (-3.5,-2);

\node[circ,label=below:{\small $T_x$ }] (a) at (-.5,0) {};
\node[circ,label=below:{\small $F_x$}] (b) at (.5,0) {};
\node[circ,label=right:{\small $u_x$}] (c) at (0,.86) {};
\node[circ,label=right:{\small $v_x$}] at (0,1.36) {};
\node[circ,label=right:{\small $w_x$}] (d) at (0,1.86) {};

\draw (a) -- (b) 
(a) -- (c) 
(b) -- (c) 
(c) -- (d);

\draw (a) -- (dc)
(a) -- (dc')
(a) -- (dc'')
(b) -- (ac)
(b) -- (ac')
(b) -- (ac'');

\draw (-2.3,-2.7) rectangle (-1.7,2.3);
\draw (1.7,-2.7) rectangle (2.3,2.3);
\end{tikzpicture}
\caption{The gadget $G_x$ for a variable $x \in X$ contained in clauses $c,c'$ and $c''$ (rectangles indicate that the corresponding set of vertices induces a clique).}
\label{fig:vargadclaw}
\end{figure}
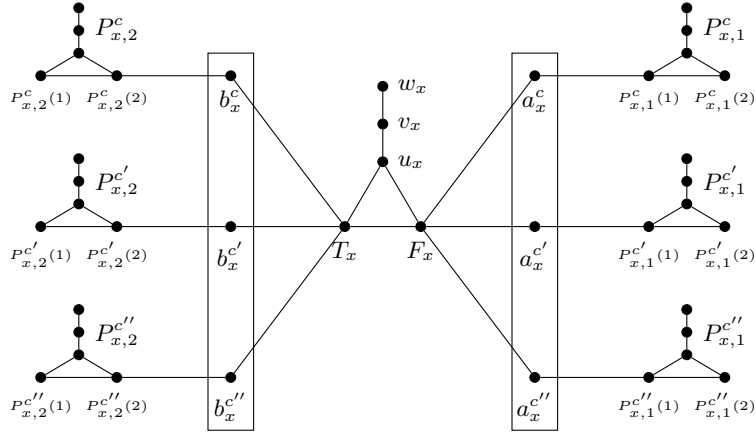

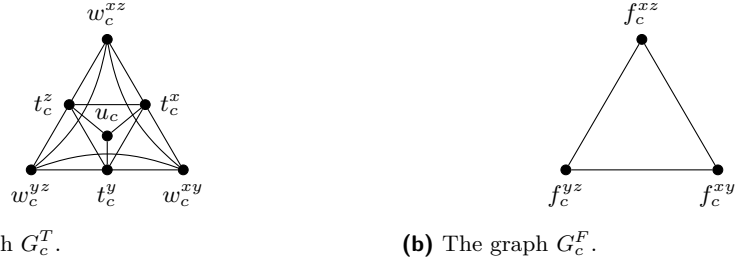
\begin{figure}[t]
\centering
\begin{subfigure}{.45\textwidth}
\centering
\begin{tikzpicture}
\node[circ,label=below:{\small $w_c^{yz}$}] (wyz) at (0,0) {};
\node[circ,label=below:{\small $w_c^{xy}$}] (wxy) at (2,0) {};
\node[circ,label=above:{\small $w_c^{xz}$}] (wxz) at (1,1.73) {};

\draw (wyz) -- (wxz) node[circ,midway,label=left:{\small $t_c^z$}]  (tz) {};
\draw (wyz) -- (wxy) node[circ,midway,label=below:{\small $t_c^y$}] (ty) {};
\draw (wxy) -- (wxz) node[circ,midway,label=right:{\small $t_c^x$}] (tx) {};

\draw (tz) -- (ty) -- (tx) -- (tz);

\draw (wyz) edge[bend right=20] (wxz);
\draw (wyz) edge[bend left=20] (wxy);
\draw (wxy) edge[bend left=20] (wxz);

\node[circ,label=above:{\small $u_c$}] (uc) at (1,.45) {};
\draw (uc) -- (tz)
(uc) -- (ty)
(uc) -- (tx); 
\end{tikzpicture}
\caption{The graph $G_c^T$.}
\end{subfigure}
\hspace*{.5cm}
\begin{subfigure}{.45\textwidth}
\centering
\begin{tikzpicture}
\node[circ,label=below:{\small $f_c^{yz}$}] (wyz) at (0,0) {};
\node[circ,label=below:{\small $f_c^{xy}$}] (wxy) at (2,0) {};
\node[circ,label=above:{\small $f_c^{xz}$}] (wxz) at (1,1.73) {};

\draw (wyz) -- (wxy) -- (wxz) -- (wyz);
\end{tikzpicture}
\caption{The graph $G_c^F$.}
\end{subfigure}
\caption{The gadget $G_c$ for a clause $c \in C$ containing variables $x,y$ and $z$.}
\label{fig:clausegadclaw}
\end{figure}

\begin{observation}
\label{obs:longpaws}
Let $D$ be a semitotal dominating set of $G$. Then for every variable $x \in X$, $|D \cap V(G_x)| \geq 14$. 
\end{observation}

Indeed for every long paw $P$ (see Figure \ref{fig:LongPaw}), the vertex $P(5)$ must be dominated and the vertex dominating $P(5)$ must have a witness; and every variable gadget contains 7 long paws. 

\begin{observation}
\label{obs:TxFx}
Let $D$ be a semitotal dominating set of $G$. If $|D \cap V(G_x)| = 14$ for some variable $x \in X$ contained in clauses $c,c'$ and $c''$, then the following holds.
\begin{itemize}
\item[1.] If $P_{x,2}^q(1) \in D$ for some $q \in \{c,c',c''\}$ then $T_x \in D$.
\item[2.] If $P_{x,1}^q(2) \in D$ for some $q \in \{c,c',c''\}$ then $F_x \in D$.
\end{itemize}
In particular, if $P_{x,2}^q(1) \in D$ for some $q \in \{c,c',c''\}$ then $D \cap \{P_{x,1}^c(2),P_{x,1}^{c'}(2),P_{x,1}^{c''}(2)\} = \emptyset$. Similarly if $P_{x,1}^q(2) \in D$ for some $q \in \{c,c',c''\}$ then $D \cap \{P_{x,2}^c(1),P_{x,2}^{c'}(1),P_{x,2}^{c''}(1)\} = \emptyset$. 
\end{observation}

Indeed, suppose that $|D \cap V(G_x)| = 14$ for some variable $x \in X$ contained in clauses $c,c'$ and $c''$. Observe first that by Observation \ref{obs:longpaws} necessarily $D \cap \{a_x^q,b_x^q~|~q \in \{c,c',c''\}\} = \emptyset$ and $|D \cap \{P_{x,j}^q(1),P_{x,j}^q(2)\}| \leq 1$ for any $j \in [2]$ and $q \in \{c,c',c''\}$ (similarly, $|D \cap \{T_x,F_x\}| \leq 1$). Thus if $P_{x,2}^q(1) \in D$ for some clause $q \in \{c,c',c''\}$ then $T_x \in D$ as $b_x^q$ must be dominated; and if $P_{x,1}^q(2) \in D$ for some $q \in \{c,c',c''\}$ then $F_x \in D$ as $a_x^q$ must be dominated.  

\begin{observation}
\label{obs:GcT}
Let $D$ be a semitotal dominating set of $G$. Then for every clause $c \in C$, $D \cap V(G_c^T) \neq \emptyset$. Furthermore, if $P_{x,2}^c(1) \notin D $ for every variable $x$ contained in $c$ then $|D \cap V(G_c^T)| \geq 2$.
\end{observation}

Indeed since $u_c$ must be dominated $D \cap V(G_c^T) \neq \emptyset$. Now if $P_{x,2}^c(1) \notin D$ for every variable $x$ contained in $c$ then the result follows from the fact that $\gamma(G_c^T) = 2$.

\begin{observation}
\label{obs:GcF}
Let $D$ be a semitotal dominating set of $G$. Then for every clause $c \in C$ containing variables $x,y$ and $z$, if $|D \cap \{P_{x,1}^c(2),P_{y,1}^c(2),P_{z,1}^c(2)\}| < 2$ then $|D \cap V(G_c^F)| \geq 1$.
\end{observation}

Indeed, if say $P_{x,1}^c(2),P_{y,1}^c(2) \notin D$ without loss of generality, then $N[f_c^{xy}] \setminus \{P_{x,1}^c(2),P_{y,1}^c(2)\} \cap D \neq \emptyset$ as $f_c^{xy}$ should be dominated.

\begin{claim}
\label{clm:phisatclaw}
$\gamma_{t2}(G) = 14|X| + |C|$ if and only if $\Phi$ is a \yes-instance for \textsc{Positive 1-In-3 3-Sat}.
\end{claim}

\begin{claimproof}
Assume first that $\Phi$ is a \yes-instance for \textsc{Positive 1-In-3 3-Sat} and consider a truth assignment satisfying $\Phi$. We construct a semitotal dominating set $D$ of $G$ as follows. For every variable $x \in X$ contained in clauses $c,c'$ and $c''$, if $x$ is set to true then add $\{T_x,v_x\} \cup \{P_{x,j}^{q}(1), P_{x,j}^{q}(4)~|~j \in [2], q \in \{c,c',c''\}\}$; and if $x$ is set to false then add $\{F_x,v_x\} \cup \{P_{x,j}^{q}(2), P_{x,j}^{q}(4)~|~j \in [2], q \in \{c,c',c''\}\}$. For every clause $c \in C$ containing variables $x,y$ and $z$, exactly one variable is set to true, say $x$ without loss of generality, in which case we add $t_c^y$ to $D$. It is not difficult to see that the constructed set $D$ is a semitotal dominating set of $G$ of size $14|X| + |C|$. We then conclude by Observations~\ref{obs:longpaws} and \ref{obs:GcT} that $D$ has minimum size.

Conversely, assume that $\gamma_{t2}(G) = 14|X| + |C|$ and consider a minimum semitotal dominating set $D$ of $G$. Note that by Observations~\ref{obs:longpaws} and \ref{obs:GcT}, $|D \cap V(G_x)| = 14$ for every variable $x \in X$ and $|D \cap V(G_c^T)| = 1$ for every clause $c \in C$; in particular, $D \cap V(G_c^F) = \emptyset$ for every clause $c \in C$. Now consider a clause $c \in C$ containing variables $x,y$ and $z$. Since $D \cap V(G_c^F) = \emptyset$, it follows from Observation~\ref{obs:GcF} that $|D \cap \{P_{x,1}^c(2),P_{y,1}^c(2),P_{z,1}^c(2)\}| \geq 2$, say $P_{x,1}^c(2),P_{y,1}^c(2) \in D$ without loss of generality. Note that then by Observation~\ref{obs:TxFx}, $F_x,F_y \in D$ (and thus $T_x,T_y \notin D$). We claim that then $P_{z,2}^c(1) \in D$. Indeed by Observation~\ref{obs:TxFx}, we have that $P_{x,2}^c(1), P_{y,2}^c(1) \notin D$. Thus if $P_{z,2}^c(1) \notin D$ then by Observation~\ref{obs:GcT}, $|D \cap V(G_c^T)| \geq 2$ a contradiction. Thus $P_{z,2}^c(1) \in D$ and so by Observation~\ref{obs:TxFx}, $T_z \in D$ (which implies that $F_z \notin D$). We thus construct a truth assignment satisfying $\Phi$ as follows: for every variable $x \in X$, if $T_x \in D$ then set $x$ to true, otherwise set $x$ to false.
\end{claimproof}

\begin{claim}
\label{clm:gammat2claw}
$\gamma_{t2}(G) = 14|X| + |C|$ if and only if $G$ is a \no-instance for \contracstd{}.
\end{claim}

\begin{claimproof}
Assume first that $\gamma_{t2}(G) = 14|X| + |C|$ and consider a minimum semitotal dominating set $D$ of $G$. Then by Observations~\ref{obs:longpaws} and \ref{obs:GcT}, $|D \cap V(G_x)| = 14$ for every variable $x \in X$ and $|D \cap V(G_c^T)| = 1$ for every clause $c \in C$; in particular, $D \cap V(G_c^F) = \emptyset$ for every clause $c \in C$. It follows that for every variable $x \in X$, $D \cap V(G_x)$ contains no friendly triple: indeed, any two distinct long paws are at distance at least 2 from one another and if some long paw $P$ contains an edge $e \in E(D)$, then $P(4)$ is an endvertex of $e$ and so $e$ is at distance at least three from any other vertex in $D \cap V(G_x)$. Now consider a clause $c \in C$ containing variables $x,y$ and $z$ and denote by $u$ the vertex in $D \cap V(G_c^T)$. Since $u$ cannot alone dominate every vertex in $V(G_c^T)$, there must exist $p \in \{x,y,z\}$ such that $P_{p,2}^c(1) \in D$, say $p=x$ without loss of generality. We claim that then $P_{y,2}^c(1),P_{z,2}^c(1) \notin D$. Indeed, if say $P_{y,2}^c(1) \in D$ then by Observation~\ref{obs:TxFx}, $P_{x,1}^c(2),P_{y,1}^c(2) \notin D$ which by Observation~\ref{obs:GcF} implies that $D \cap V(G_c^F) \neq \emptyset$, a contradiction. Thus $P_{y,2}^c(1) \notin D$ and we conclude similarly that $P_{z,2}^c(1) \notin D$. But then $u \notin \{w_c^{xz},w_c^{xy},t_c^x\}$: indeed, if $u = t_c^x$ then $w_c^{yz}$ is not dominated, and if $u \in \{w_c^{xz},w_c^{xy}\}$ then $u_c$ is not dominated. It follows that $u$ is at distance at least two from $P_{x,2}^c(1)$ and can thus not be part of a friendly triple. Hence, $D$ contains no friendly triple and so, $G$ is a \no-instance for \contracstd{} by Theorem~\ref{thm:friendlytriple}(i).

Conversely, assume that $G$ is a \no-instance for \contracstd{} and consider a minimum semitotal dominating set $D$ of $G$. Observe first that if $|D \cap V(P)| \geq 3$ for some long paw $P$ then clearly $D \cap V(P)$ contains a friendly triple. Thus for every variable $x \in X$ contained in clauses $c,c'$ and $c''$, $|D \cap V(P_{x,j}^q)| \leq 2$ for every $j \in [2]$ and $q \in \{c,c',c''\}$ (similarly, $|D \cap \{T_x,F_x,u_x,v_x,w_x\}| \leq 2$). By Observation \ref{obs:longpaws}, we conclude that in fact equality holds. We may further assume that $P(5) \notin D$ for every long paw $P$ of 
$G_x$ (consider otherwise $(D \setminus \{P(5),P(4),P(3)\}) \cup \{P(3),P(4)\}$) which implies in particular that every vertex of a long paw $P$ is dominated by some vertex in $D \cap V(P)$. It follows that for any $q \in \{c,c',c''\}$,  $b_x^q \notin D$: indeed, if $b_x^q \in D$ then $T_x \notin D$ ($D$ would otherwise contain a friendly triple, namely $b_x^q,T_x,v_x$) and so $D'=(D \setminus \{b_x^q\}) \cup  \{T_x\}$ is a minimum semitotal dominating set of $G$ containing a friendly triple, namely $D' \cap \{T_x,F_x,u_x,v_x,w_x\}$, a contradiction. By symmetry, we conclude that $a_x^q \notin D$ for any $q \in \{c,c',c''\}$ and so, $|D \cap V(G_x)| = 14$. Now consider a clause $c \in C$ containing variables $x,y$ and $z$. Suppose first that $|D \cap V(G_c^T)| \geq 2$. Then $D \cap \{P_{x,2}^c(1),P_{y,2}^c(1),P_{z,2}^c(1)\} = \emptyset$: indeed, if say $P_{x,2}^c(1) \in D$ then $(D \setminus V(G_c^T)) \cup \{t_c^x,t_c^y\}$ is a semitotal dominating set of $G$ of size at most that of $D$ containing a friendly triple, namely $t_c^x,t_c^y,P_{x,2}^c(1)$, a contradiction. Thus $P_{x,2}^c(1) \notin D$ and we conclude similarly that $P_{y,2}^c(1),P_{z,2}^c(1) \notin D$. But then $D' = (D \setminus V(G_c^T)) \cup \{t_c^y,P_{x,2}^c(1)\}$ is a semitotal dominating set of $G$ of size at most that of $D$ containing a friendly triple, namely $D' \cap  V(P_{x,2}^c)$, a contradiction. Thus $|D \cap V(G_c^T)| \leq 1$ and we conclude by Observation~\ref{obs:GcT} that in fact equality holds. Second, observe that if $|D \cap V(G_c^F)| \geq 2$, say $f_c^{yz},f_c^{xy} \in D$ without loss of generality, then $D$ contains a friendly triple as $D \cap \{P_{x,1}^c(1),P_{x,1}^c(2),P_{x,1}^c(3)\} \neq \emptyset$ by the above, a contradiction. Thus suppose that $|D \cap V(G_c^F)| =1$, say $f_c^{xy} \in D$. Then $P_{x,1}^c(2),P_{y,1}^c(2) \notin D$: indeed, if $P_{p,1}^c(2) \in D$ for some $p \in \{x,y\}$ then $f_c^{xy} \cup (D \cap V(P_{p,1}^c))$ contains a friendly triple, a contradiction. It follows that $F_x \notin D$ for otherwise $(D \setminus V(P_{x,1}^c)) \cup \{P_{x,1}^c(2),P_{x,1}^c(4)\}$ is a minimum semitotal dominating set of $G$ containing a friendly triple, namely $P_{x,1}^c(2),f_c^{xy}, D \cap \{P_{y,1}^c(1),P_{y,1}^c(3)\}$, a contradiction. We conclude similarly that $F_y \notin D$. But then we may assume that $T_x,T_y \in D$ (consider otherwise $(D \setminus \{u_x,v_x,w_x,u_y,v_y,w_y\}) \cup \{T_x,v_x,T_y,v_y\}$) and that $P_{x,2}^c(1),P_{y,2}^c(1) \in D$ (consider otherwise $(D \setminus (V(P_{x,2}^c) \cup V(P_{y,2}^c))) \cup \{P_{p,2}(1),P_{p,2}(4)~|~p \in \{x,y\}\}$). But then $(D \setminus V(G_c^T)) \cup \{t_c^x\}$ is a minimum semitotal dominating set of $G$ containing a friendly triple, namely $P_{y,2}^c(1),t_c^x,P_{x,2}^c(1)$, a contradiction. Thus $D \cap V(G_c^F) = \emptyset$ which implies that $|D \cap V(G_c)| = 1$. Therefore $|D| = 14|X| + |C|$, which concludes the proof.
\end{claimproof}

By combining Claims~\ref{clm:phisatclaw} and \ref{clm:gammat2claw}, we obtain that $G$ is a \no-instance for \contracstd{} if and only if $\Phi$ is a \yes-instance for \textsc{Positive 1-In-3 3-Sat}. As it is not difficult to see that $G$ is claw-free, this concludes the proof.

\section{Proof of Theorem~\ref{thm:2P3}}

We introduce an auxiliary problem which will be helpful in showing the $\mathsf{coNP}$-hardness of \contracstd{} in $2P_3$-free graphs.

\begin{center}
\fbox{
\begin{minipage}{5.5in}
\textsc{All Independent MSD}
\begin{description}
\item[Instance:] A graph $G$.
\item[Question:] Is every minimum semitotal dominating set of $G$ independent?
\end{description}
\end{minipage}}
\end{center}

In the following hardness proof, we reduce from the \textsc{Positive 1-In-3 3-Sat} problem which is a variant of the \textsc{3-Sat} problem where given a formula $\Phi$ in which all literals are positive, the problem is to determine whether there exists a truth assignment such that each clause has exactly one true literal. This problem was shown to be $\mathsf{NP}$-complete in \cite{schaefer}.

\begin{lemma}
\label{lem:allindepMSD}
\textsc{All Independent MSD} is $\mathsf{NP}$-hard when restricted to $2P_3$-free graphs.
\end{lemma}

\begin{proof}
We reduce from \textsc{Positive 1-In-3 3-Sat}: given a instance $\Phi$ of this problem, with variable set $X$ and clause set $C$, we construct an equivalent instance $G_{\Phi}$ of \textsc{All Independent MSD} as follows. For every variable $x \in X$, we introduce a triangle $G_x$ which has two distinguished \emph{truth vertices} $T_x$ and $F_x$ (we denote by $u_x$ the third vertex of $G_x$). For every clause $c \in C$ containing variables $x,y,z$, we introduce a $K_5$ denoted by $G_c$ with vertex set $\{v_c^x,v_c^y,v_c^z,u_c^T,u_c^F\}$. The adjacencies between the gadgets are as follows.
\begin{itemize}
    \item[$\cdot$] For every clause $c \in C$ containing variables $x,y,z$, we connect $u_c^T$ to $T_x,T_y,T_z$ and $u_c^F$ to $F_x,F_y,F_z$; we further connect $v_c^s$ to $T_s$ and $F_r$ for every $s\in \{x,y,z\}$ and every $r \in \{x,y,z\} \setminus \{s\}$.
    \item[$\cdot$] $\bigcup_{c \in C} V(G_c)$ induces a clique.
\end{itemize}
We denote by $G_{\Phi}$ the resulting graph.\\

Since $u_x$ must be dominated in any semitotal dominating set, we trivially have the following observation.

\begin{observation}
\label{obs:varGx}
Let $D$ be a semitotal dominating set of $G_{\Phi}$. Then $|D \cap V(G_x)| \geq 1$ for every variable $x \in X$.
\end{observation}

\begin{claim}
\label{clm:phisatind}
$\gamma_{t2}(G_{\Phi}) = |X|$ if and only if $\Phi$ is satisfiable.
\end{claim}

\begin{claimproof}
Assume that $\Phi$ is satisfiable and consider a truth assignment satisfying $\Phi$. We construct a semitotal dominating set $D$ of $G_{\Phi}$ as follows. For every variable $x \in X$, if $x$ is set to true then we add $T_x$ to $D$, otherwise we add $F_x$ to $D$. Clearly every variable gadget is dominated by some vertex in $D$. Now consider a clause $c$ containing variables $x,y,z$. Then, exactly one variable is set to true, say $x$ without loss of generality.  Since $\{T_x,F_y,F_z\} \subset D$, $v_c^x$ and $u_c^T$ are dominated by $T_x$, $v_c^y$ and $u_c^F$ are dominated by $F_z$ and $v_c^z$ is dominated by $F_y$. Furthermore, $T_x,F_y,F_z$ are pairwise at distance exactly two ($v_c^x$ is a common neighbour). Thus $D$ is a semitotal dominating set of $G_{\Phi}$ and has minimum size by Observation~\ref{obs:varGx}.

Conversely, assume that $\gamma_{t2}(G_{\Phi}) = |X|$ and let $D$ be a minimum semitotal dominating set of $G_{\Phi}$. Then by Observation~\ref{obs:varGx}, $|D \cap V(G_x)| =1$ for every variable $x \in X$ which in turn implies that $D \cap \bigcup_{c \in C} V(G_c) = \emptyset$. It follows that for any variable $x \in X$, $u_x \notin D$: indeed, if $u_x \in D$ for some variable $x \in X$ then $u_x$ has no witness as $D \cap (\{T_x,F_x\} \cup \bigcup_{c \in C} V(G_c)) = \emptyset$, a contradiction. Now consider a clause $c \in C$ containing variables $x,y,z$ and suppose that there exist two variables $s,r \in \{x,y,z\}$ such that $\{T_s,T_r\} \subset D$. Then one of $u_c^F$ and $v_c^q$ where $q \in \{x,y,z\} \setminus \{s,r\}$ is not dominated: indeed, either $T_q \in D$ in which case $u_c^F$ is not dominated, or $F_q \in D$ in which case $v_c^q$ is not dominated. Thus there exists at most one variable $s \in \{x,y,z\}$ such that $T_s \in D$ and since $u_c^T$ must be dominated, we conclude that such a variable exists. Thus, the truth assignment constructed by setting $x$ to true if $T_x \in D$ and $x$ to false if $F_x \in D$ satisfies $\Phi$.
\end{claimproof}

\begin{claim}
\label{clm:allind}
$\gamma_{t2}(G_{\Phi}) = |X|$ if and only if $G_{\Phi}$ is a \yes-instance for \textsc{All Independent MSD}.
\end{claim}

\begin{claimproof}
Assume that $\gamma_{t2}(G_{\Phi}) = |X|$ and let $D$ be an arbitrary minimum semitotal dominating set of $G_{\Phi}$. Then by Observation~\ref{obs:varGx}, $|D \cap V(G_x)| = 1$ for any variable $x \in X$ which implies that $D \cap \bigcup_{c \in C} V(G_c) = \emptyset$. Thus $D$ is independent and so $G_{\Phi}$ is a \yes-instance for \textsc{All Independent MSD}.

Conversely, assume that $G_{\Phi}$ is a \yes-instance for \textsc{All Independent MSD} and let $D$ be a minimum semitotal dominating set of $G_{\Phi}$. Since $D$ is independent, $|D \cap V(G_x)| \leq 1$ for any variable $x \in X$ and we conclude by Observation~\ref{obs:varGx} that in fact equality holds. Furthermore, we may assume that for any variable $x \in X$, $u_x \notin D$ as it suffices to consider $(D \setminus \{u_x\}) \cup \{T_x\}$ otherwise. It follows that if two variables $x$ and $y$ both occur in some clause $c$ and $r \in D \cap V(G_x)$ and $s \in D \cap V(G_y)$, then $r$ and $s$ witness each other: indeed, $d(T_x,T_y) = d(F_x,F_y) = d(T_x,F_y) = d(T_y,F_x) =2$. Now consider a clause $c \in C$ containing variables $x,y,z$, and suppose to the contrary that there exists $w \in V(G_c)\cap D$. Since $D$ is independent, $(D \cap \bigcup_{c'\in C} V(G_{c'}))= \set{w}$ (recall that $\bigcup_{c \in C} V(G_c)$ induces a clique). Furthermore, by the previous observation, any vertex $t \in D$ witnessed by $w$ is also witnessed by a vertex in $D \setminus \{w\}$. Thus we may replace $w$ with either $u_c^T$ or $u_c^F$ and obtain a semitotal dominating set of $G_{\Phi}$ which is not independent: if $D \cap \{F_x,F_y,F_z\} \neq \emptyset$ then $(D \setminus \{w\}) \cup \{u_c^F\}$ is a minimum semitotal dominating set of $G_{\Phi}$ which is not independent, otherwise $(D \setminus \{w\}) \cup \{u_c^T\}$ is a minimum semitotal dominating set of $G_{\Phi}$ which is not independent. Since this contradicts the fact that $G_{\Phi}$ is a \yes-instance for \textsc{All Independent MSD}, we conclude that $D \cap \bigcup_{c \in C} V(G_c) = \emptyset$ and so $\gamma_{t2}(G_{\Phi}) = |D| = |X|$. 
\end{claimproof}

By combining Claims~\ref{clm:phisatind} and \ref{clm:allind}, we obtain that $\Phi$ is satisfiable if and only if $G_{\Phi}$ is a \yes-instance for \textsc{All Independent MSD}. Finally, it is not difficult to see that $G_{\Phi}$ is $2P_3$-free, which concludes the proof.
\end{proof}

\begin{lemma}
\label{lem:2P3Ind}
Let $G$ be a $2P_3$-free graph. Then $G$ is a \yes-instance for \textsc{1-Edge Contrac-tion($\gamma_{t2}$)} if and only if $G$ is a \no-instance for \textsc{All Independent MSD}.
\end{lemma}

\begin{proof}
If $G$ is a \yes-instance for \contracstd{} then by Theorem~\ref{thm:friendlytriple}(i) $G$ has a minimum semitotal dominating set containing a friendly triple which a fortiori is not independent. Thus, $G$ is a \no-instance for \textsc{All Independent MSD}.

Conversely, assume that $G$ is a \no-instance for \textsc{All Independent MSD} and let $D$ be a minimum semitotal dominating set of $G$ which is not independent. Suppose that $D$ contains no friendly triple and let $x,y \in D$ be two adjacent vertices. Consider a vertex $z \in D\setminus \{x,y\}$ such that $d(z,\{x,y\}) = \min_{u \in D \setminus \{x,y\}} d(u,\{x,y\})$ and assume without loss of generality that $d(z,\{x,y\}) = d(z,x)$. By assumption $d(x,z) >2$ and since $G$ is $2P_3$-free, $d(x,z) \leq 5$.

Suppose first that $d(x,z) = 3$ and let $P= xuvz$ be a shortest path from $x$ to $z$. Let $w \in D$ be a witness for $z$. Suppose first that $w$ is adjacent to $z$. If $y$ has no private neighbour then $(D \setminus \{y\}) \cup \{u\}$ is a minimum semitotal dominating set of $G$ (indeed, since $D$ contains no friendly triple by assumption, $y$ is a witness for $x$ only) containing a friendly triple, namely $x,u,z$. We conclude similarly if $w$ has no private neighbour. Thus we may assume that both $y$ and $w$ have at least one private neighbour, say $p_y$ and $p_w$ respectively. But then the private neighbourhood of $w$ must be complete to the private neighbourhood of $y$: indeed, if $y$ has a private neighbour $a$ and $w$ has a private neighbour $b$ such that $a$ and $b$ are nonadjacent then $\{a,y,x,b,w,z\}$ induces a $2P_3$, a contradiction. It follows that $(D \setminus \{y,w\}) \cup \{p_y,p_w\}$ is a minimum semitotal dominating set containing a friendly triple, namely $p_y,p_w,x$. Second, suppose that $d(z,w) = 2$. If $w$ is adjacent to $v$ then every private neighbour of $y$ is adjacent to $v$: indeed, if $y$ has a private neighbour $p_y$ which is non adjacent to $v$ then $\{p_y,y,x,z,v,w\}$ induces a $2P_3$, a contradiction. But then $(D \setminus \{y\}) \cup \{v\}$ is a minimum semitotal dominating set containing a friendly triple, namely $z,v,w$. Thus, assume that $w$ is nonadjacent to $v$ and let $t$ be the internal vertex in a shortest path from $z$ to $w$. If $x$ has no private neighbour then necessarily $y$ is adjacent to $u$ ($u$ would otherwise be a private neighbour of $x$) and so the minimum dominating set $(D \setminus \{x\}) \cup \{u\}$ contains a friendly triple, namely $y,u,z$. Thus, assume that $x$ has at least one private neighbour. Then every private neighbour $p_x$ of $x$ must be adjacent to $t$ for otherwise $\{p_x,x,y,z,t,w\}$ induces a $2P_3$; in particular, $x$ is at distance two from $t$. Similarly, we conclude that every private neighbour of $y$ is adjacent to $t$. It then follows that $(D \setminus \{y\}) \cup \{t\}$ is a minimum semitotal dominating set of $G$ containing a friendly triple, namely $z,t,w$. 

Suppose next that $d(x,z) = 4$ and let $P = xuvtz$ be a shortest path from $x$ to $z$. Let $w \in D$ be a witness for $z$. Suppose first that $w$ is adjacent to $z$. We claim that either $y$ has no private neighbour of $w$ has no private neigbor. Indeed, if $y$ has a private neighbour $p_y$ and $w$ has a private neighbour $p_w$ then $p_y$ and $p_w$ must be adjacent for otherwise $\{p_y,y,x,p_w,w,z\}$ induces a $2P_3$, a contradiction. But then $d(y,w) \leq 3 < d(x,z)$ a contradiction to our assumption. Thus assume without loss of generality that $y$ has no private neighbour. Then it suffices to consider $(D \setminus \{y\}) \cup \{u\}$ and go back to the previous case. Second, suppose that $d(z,w) = 2$ and let $q$ be the internal vertex in a shortest path from $z$ to $w$. Then $y$ has no private neighbour: indeed if $y$ has a private neighbour $a$ then $a$ is adjacent to $q$ ($\{a,y,x,w,q,z\}$ would otherwise induce a $2P_3$) which implies that $d(y,z) \leq 3 < d(x,z)$, a contradiction to our assumption. But then it suffices to consider $(D \setminus\{y\}) \cup \{u\}$ and go back to the previous case. 

Suppose finally that $d(x,z) = 5$ and let $P = u_1\ldots u_6$ where $u_1=x$ and $u_6=z$, be a shortest path from $x$ to $z$. Then $y$ has no private neighbour: indeed, if $y$ has a private neighbour $a$ then $a$ is adjacent to either $u_4$ or $u_5$ (since $\{a,y,x,u_4,u_5,z\}$ would otherwise induce a $2P_3$) and so $d(y,z) \leq 4 < d(x,z)$, a contradiction to our assumption. But then it suffices to consider $(D \setminus \{y\}) \cup \{u_2\}$ and go back to the previous case, which concludes the proof.
\end{proof}

Theorem~\ref{thm:2P3} now follows from Lemmas~\ref{lem:allindepMSD} and \ref{lem:2P3Ind}.

\section{Proof of Theorem~\ref{thm:P6P4+P2}}

We use the same construction as in \cite[Theorem 3.1]{domcontract}: given an instance $(G,\ell)$ of \textsc{Dominating Set}, we construct an equivalent instance $G'$ of \contracstd{} as follows. We denote by $\{v_1,\ldots,v_n\}$ the vertex set of $G$. The vertex set of the graph $G'$ is given by $V(G')=V_0\cup\ldots\cup V_\ell\cup\{x_0,\ldots,x_\ell,y\}$, where each $V_i$ is a copy of the vertex set of $G$. We denote the vertices of $V_i$ by $v^i_1,v^i_2,\ldots,v^i_n$. The adjacencies in $G'$ are then defined as follows:
\begin{itemize}
\item[$\cdot$] $V_0\cup\{x_0\}$ is a clique;
\item[$\cdot$] $yx_0\in E(G')$;
\end{itemize}
and for $1\leq i\leq \ell$,
\begin{itemize}
\item[$\cdot$] $V_i$ is an independent set;
\item[$\cdot$] $x_i$ is adjacent to all the vertices of $V_0\cup V_i$;
\item[$\cdot$] $v^i_j$ is adjacent to $\{v^0_a~|~v_a\in N_G[v_j]\}$ for any $1 \leq j \leq n$.
\end{itemize}

\begin{figure}[t]
\centering
\begin{tikzpicture}[scale=.5]
\draw (0,0) ellipse (2cm and .8cm);
\node[draw=none] at (0,0) {$V_0$};

\node[circ,label=above left:{\small $x_1$}] (x1) at (2.21,1.27) {};
\draw[rotate=37] (4.5,0) ellipse (1cm and .5cm);
\node[draw=none] at (3.59,2.71) {$V_1$};
\draw[very thick] (x1) -- (1.99,.07)
(x1) -- (.17,.8)
(x1) -- (2.74,2.34)
(x1) -- (3.18,1.99);

\node[circ,label=below left:{\small $x_2$}] (x2) at (2.21,-1.27) {};
\draw[rotate=-37] (4.5,0) ellipse (1cm and .5cm);
\node[draw=none] at (3.59,-2.71) {$V_2$};
\draw[very thick] (x2) -- (1.99,-.07)
(x2) -- (.17,-.8)
(x2) -- (2.74,-2.34)
(x2) -- (3.18,-1.99);

\node[draw=none] at (0,-1.8) {$\ldots$};

\node[circ,label=below right:{\small $x_{\ell}$}] (xl) at (-2.21,-1.27) {};
\draw[rotate=37] (-4.5,0) ellipse (1cm and .5cm);
\node[draw=none] at (-3.59,-2.71) {$V_{\ell}$};
\draw[very thick] (xl) -- (-1.99,-.07)
(xl) -- (-.17,-.8)
(xl) -- (-2.74,-2.34)
(xl) -- (-3.18,-1.99);

\node[circ,label=above right:{\small $x_0$}] (xl1) at (-2.21,1.27) {};
\node[circ,label=left:{\small $y$}] (y) at (-3,2) {};
\draw[very thick] (xl1) -- (-1.99,.07)
(xl1) -- (-.17,.8);
\draw[-] (xl1) -- (y);
\end{tikzpicture}
\caption{The graph $G'$ (thick lines indicate that the vertex $x_i$ is adjacent to every vertex in $V_0$ and $V_i$, for $i=0,\ldots,\ell$).}
\label{fig:dom1ec}
\end{figure}
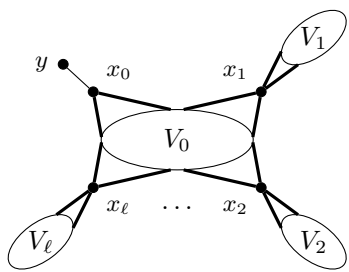

\begin{claim}
\label{claim:gammaG'bis}
$\gamma_{t2}(G') = \min \{\gamma (G) + 1, \ell + 1\}$. 
\end{claim}

\begin{claimproof}
It is clear that $\{x_0,x_1, \ldots, x_\ell\}$ is a semitotal dominating set of $G'$ and so $\gamma_{t2}(G') \leq \ell +1$. If $\gamma (G) \leq \ell$ and $\{v_{i_1},\ldots, v_{i_p}\}$ is a minimum dominating set of $G$, it is easily seen that $\{v^0_{i_1},\ldots, v^0_{i_p},x_0\}$ is a semitotal dominating set of $G'$. Thus, $\gamma_{t2}(G') \leq \gamma (G) + 1$ and so, $\gamma_{t2}(G') \leq \min \{\gamma (G) + 1, \ell + 1\}$. Now, suppose to the contrary that $\gamma_{t2}(G') < \min \{\gamma (G) + 1, \ell + 1\}$ and consider a minimum semitotal dominating set $D'$ of $G'$. We first make the following simple observation.

\begin{observation}
\label{obs:yxl+1bis}
For any semitotal dominating set $D$ of $G'$, $D \cap \{y,x_0\} \neq~\emptyset$.
\end{observation}

Now, since $\gamma_{t2}(G') < \ell + 1$, there exists $1 \leq i \leq \ell$ such that $x_i \not\in D'$ (otherwise, $\{x_1,\ldots,x_{\ell}\} \subset D'$ and combined with Observation \ref{obs:yxl+1bis}, $D'$ would be of size at least $\ell + 1$). But then, $D'' = D' \cap (V_0 \cup V_i)$ must dominate every vertex in $V_i$, and so $\vert D'' \vert \geq \gamma(G)$. Since $\vert D'' \vert \leq \vert D' \vert - 1$ (recall that $D' \cap \{y,x_0\} \neq \emptyset$), we then have $\gamma (G) \leq \vert D' \vert - 1$, a contradiction. Thus, $\gamma_{t2}(G') = \min \{\gamma (G) + 1, \ell + 1\}$.
\end{claimproof}

We now show that $(G,\ell)$ is a \yes-instance for \textsc{Dominating Set} with $\gamma(G) \geq 2$ if and only if $G'$ is a \yes-instance for \contracstd{}.

Assume first that $\gamma (G) \leq \ell$. Then $\gamma_{t2}(G') = \gamma(G) + 1$ by the previous claim, and if $\{v_{i_1},\ldots,v_{i_p}\}$ is a minimum dominating set of $G$, then $\{v^0_{i_1},\ldots,v^0_{i_p},x_0\}$ is a minimum semitotal dominating set of $G'$ containing a friendly triple (recall that we assume that $\gamma(G) \geq 2$). Hence, by Theorem \ref{thm:friendlytriple}(i), $G'$ is a \yes-instance for \contracstd{}.

Conversely, assume that $G'$ is a \yes-instance for \contracstd{} that is, there exists a minimum semitotal dominating set $D'$ of $G'$ containing a friendly triple (see Theorem \ref{thm:friendlytriple}(i)), say $x,y,z$ where $x$ and $y$ are adjacent and $d_{G'}(y,z) \leq 2$. Then, Observation \ref{obs:yxl+1bis} implies that there exists $1 \leq i \leq \ell$ such that $x_i \not\in D'$; indeed, if it weren't the case, we would then have by Claim \ref{claim:gammaG'bis} $\gamma (G') = \ell + 1$ and thus, $D'$ would consist of $x_1, \ldots, x_{\ell}$ and either $y$ or $x_0$. In both cases, $D'$ would not contain a friendly triple, a contradiction. It follows that $D'' = D' \cap (V_0 \cup V_i)$ must dominate every vertex in $V_i$ and thus, $\vert D'' \vert \geq \gamma (G)$. But $\vert D'' \vert \leq \vert D' \vert -1$ (recall that $D' \cap \{y,x_0\} \neq \emptyset$) and so by Claim \ref{claim:gammaG'bis}, $\gamma (G) \leq \vert D' \vert - 1 \leq (\ell + 1) - 1$ that is, $(G,\ell)$ is a \yes-instance for \textsc{Dominating Set}.\\

Since it was shown in \cite[Theorem 3.1]{domcontract} that $G'$ is a $\{P_6,P_4+P_2\}$-free chordal graph, the result follows.

\section{Proof of Claim~\ref{clm:Rempty}}

We here turn to the proof of Claim~\ref{clm:Rempty}. Assume henceforth that $\mathcal{R} = \varnothing$. We first prove the following claims.

\begin{claim}\label{ComponentsD}
If $G$ is a \no-instance for \contracstd{} and $D$ is a minimum semitotal dominating set of $G$, then every connected component of $G[D]$ has cardinality at most two and there are at most $|A|$ components of cardinality 2.
\end{claim}

\begin{claimproof}
The first claim follows from Theorem~\ref{thm:friendlytriple}(i) since any connected component of size at least three in $G[D]$ would contain a friendly triple. For the second claim, since any component of size two has to be at distance at least three to every other component ($D$ would otherwise contain a friendly triple), every vertex of $A$ can be adjacent to at most one component of size two; and every size-two-component  $C_0$ of $G[D]$ has to be adjacent to at least one vertex in $A$ for otherwise $A \cup C_0$ would induce a $P_3+kP_2$.
\end{claimproof}

Assume henceforth that $G$ is a \no-instance. In the following, given a minimum semitotal dominating $D$ of $G$, we denote by $D' \subseteq D$ the set of size-one components in $D$ (note that $D'$ is an independent set).

\begin{claim}\label{OnlyOnePNInC}
Let $D$ be a minimum semitotal dominating set of $G$. If there exists a vertex $b\in B\cap D'$ which has more than one private neighbour in $C$ then $|B\cap D'|\leq k|A|$.
\end{claim}

\begin{claimproof}
Assume that there exists a vertex $b\in B\cap D'$ which has at least two private neighbours in $C$, say $x$ and $y$. Suppose for a contradiction that there are at least $k|A|$ further vertices in $B\cap D'$ besides $b$, say $b_1,\ldots,b_{k|A|}$. For every $i\in[k]$ there has to be a vertex $c_i\in C$ such that $N(c_i)\cap D\subseteq \{b_{(i-1)|A|+1},\ldots,b_{i|A|}\}$: indeed, if for some $i \in [k]$, no such vertex in $C$ exists then $(D\setminus \{b_{(i-1)|A|+1},\ldots,b_{i|A|}\})\cup A$ is a minimum semitotal dominating set containing a friendly triple (note indeed that $A$ dominates all of $A\cup B$ and every vertex is within distance at most two of a vertex in $A$), a contradiction to Theorem~\ref{thm:friendlytriple}. Thus assume, without loss of generality, that $b_{i|A|}$ is adjacent to $c_i$ for every $i\in[k]$. Then the vertices $b,x,y,c_1,\ldots,c_k,b_{|A|},b_{2|A|},\ldots,b_{k|A|}$ induce a $P_3+kP_2$, a contradiction.
\end{claimproof}

\begin{claim}\label{SharedResponsibilitiesAreDistributed}
Let $D$ be a minimum semitotal dominating set of $G$. If there exists a vertex $c\in C$ such that $\vert N(c)\cap D'\vert\geq 2$ then $c$ is adjacent to all vertices in $B\cap D'$ except for at most $k|A|-1$.
\end{claim}

\begin{claimproof}
Assume that $c\in C$ has at least two neighbours $x,y\in B\cap D'$. Suppose for a contradiction that there are at least $k\vert A\vert$ vertices in $B\cap D'$ which are not adjacent to $c$, say $b_1,\ldots,b_{k|A|}$. As shown in the proof of Claim~\ref{OnlyOnePNInC}, there then has to be for every $i \in [k]$ a vertex $c_i \in C$ such that $N(c_i)\cap D\subseteq \{b_{(i-1)|A|+1},\ldots,b_{i|A|}\}$. Assume, without loss of generality, that $b_{i|A|}$ is adjacent to $c_i$ for every $i\in[k]$. Then the vertices $c,x,y,c_1,\ldots,c_k,b_{|A|},b_{2|A|},\ldots,b_{k|A|}$ induce a $P_3+kP_2$, a contradiction.
\end{claimproof}

\begin{claim}\label{FewVerticesInCWithoutPN}
Let $D$ be a minimum semitotal dominating set of $G$. If there are $|A|$ vertices in $B\cap D'$ which do not have a private neighbour in $C$ then $|B\cap D'|\leq (k+1)|A|-1$.
\end{claim}

\begin{claimproof}
Assume that $\vert B\cap D'\vert\geq (k+1)\vert A\vert$. Suppose for a contradiction that there are $|A|$ vertices, say $b_1,\ldots,b_{|A|}\in B\cap D'$, which have no private neighbours in $C$. Then for every $i\in[\vert A\vert]$, any vertex $c\in N(b_i)\cap C$ has to be adjacent to at least two vertices in $B\cap D'$ and thus, by Claim~\ref{SharedResponsibilitiesAreDistributed}, $c$ has to be adjacent to at least $\vert A\vert+1$ vertices in $B\cap D'$. But then $(D\setminus \{b_1,\ldots,b_{|A|}\})\cup A$ is a minimum semitotal dominating set containing a friendly triple, a contradiction to Theorem~\ref{thm:friendlytriple}(i).
\end{claimproof}

\begin{claim}\label{BCapDIsSmall}
Let $D$ be a minimum semitotal dominating set of $G$. If there exists a vertex $v\in B\cap D'$ which has a private neighbour $c\in N(v)\cap C$ and a private neighbour $b\in N(v)$ such that $c$ is not adjacent to $b$ then $|B\cap D'|\leq (k+1)|A|$.
\end{claim}

\begin{claimproof}
If there exists a vertex in $B\cap D'$ with two private neighbours in $C$ then we conclude by Claim~\ref{OnlyOnePNInC}. Thus, we may assume that no vertex in $B\cap D'$ has more than one private neighbour in $C$. Assume that there exists a vertex $v\in B\cap D'$ which has exactly one private neighbour $c\in C$ and another private neighbour $b\in B \cup A$ such that $b$ and $c$ are not adjacent. Suppose for a contradiction that $|B\cap D'|\geq (k+1)|A|+1$. It follows from Claim~\ref{FewVerticesInCWithoutPN} that there are at most $|A|-1$ vertices in $B\cap D'$ which do not have a private neighbour in $C$. Hence, besides $v$, there are at least $k\vert A\vert+1$ further vertices in $B\cap D'$ which do have private neighbours in $C$. Let $b_1,\ldots,b_{|A|+k}\in B\cap D'$ be $\vert A\vert+k$ of them and let $c_1,\ldots,c_{|A|+k}\in C$ be their private neighbours, respectively. By the pigeonhole principle, there are either $k$ indices $i\in[\vert A\vert+k]$ such that $c_i$ is nonadjacent to $b$ or $|A|+1$ indices $i\in[\vert A\vert+k]$ such that $c_i$ is adjacent to $b$. In the first case, assume, without loss of generality, that $c_1,\ldots,c_k$ are nonadjacent to $b$. Then $c,v,b,b_1,\ldots,b_k,c_1,\ldots,c_k $ induce a $P_3+kP_2$, a contradiction. In the second case, assume, without loss of generality,  that $b$ is complete to $\{c_1,\ldots,c_{\vert A\vert+1}\}$. Together with Claim~\ref{SharedResponsibilitiesAreDistributed}, we then conclude that every vertex in $C$ which is adjacent to a vertex in $\{b_1,\ldots,b_{\vert A\vert+1}\}$ is adjacent to a vertex in $((B\cap D)\setminus \{b_1,\ldots,b_{\vert A\vert+1}\})\cup \{b\}$ as well, and so $(D\setminus \{b_1,\ldots,b_{|A|+1}\})\cup \{b\}\cup A$ yields a minimum semitotal dominating set containing a friendly triple, a contradiction to Theorem~\ref{thm:friendlytriple}(i).
\end{claimproof}

\begin{claim}\label{TwoneighboursInD'}
Let $D$ be a minimum semitotal dominating set of $G$. If there exists a vertex $v \in V(G) \setminus D$ such that $v$ has exactly two neighbours in $D'$ then $|B \cap D'| \leq (k+1)(|A|+1)-1$.
\end{claim}

\begin{claimproof}
Assume that there exists a vertex $v \in V(G) \setminus D$ such that $v$ has exactly two neighbours in $D'$, say $b$ and $b'$. Suppose to the contrary that $|B \cap D'| \geq (k+1)(|A|+1)$. Then by Claim~\ref{OnlyOnePNInC}, every vertex in $B\cap D'$ has at most one private neighbour in $C$; and Claim~\ref{FewVerticesInCWithoutPN} ensures that there are at least $k(|A|+1)+2$ vertices $b_1,\ldots,b_{k(|A|+1)+2}\in B\cap D'$ which do have a private neighbour in $C$, say $c_1,\ldots,c_{k(|A|+1)+2}\in C$ respectively. Assume without loss of generality that $b$ and $b'$ are distinct from $b_1,\ldots,b_{k(|A|+1)}$. Then there exist at least $k|A|+1$ indices $i \in [k(|A|+1)]$ such that $c_i$ is adjacent to $v$: indeed, if there are at most $k|A|$ such indices then there are at least $k$ indices $i \in [k(|A|+1)]$ such that $c_i$ is nonadjacent to $v$, say indices 1 through $k$ without loss of generality. But then the vertices $b,v,b',b_1,\ldots,b_k,c_1,\ldots,c_k$ induce a $P_3+kP_2$, a contradiction. Thus assume, without loss of generality, that $c_i$ is adjacent to $v$ for every $i\in [k|A|+1]$. Then $(D \setminus \{b_1,\ldots,b_{|A|+1}\}) \cup \{v\} \cup A$ is a minimum semitotal dominating set of $G$ containing a friendly triple, a contradiction to Theorem~\ref{thm:friendlytriple}(i).
\end{claimproof}

\begin{claim}\label{MostVerticesInC}
There exists a minimum semitotal dominating set $D$ of $G$ with a maximum number of size-two components amongst all minimum semitotal dominating sets of $G$, such that the number of size-one components outside of $C$ is at most $(k+1)(\vert A\vert +2)+2\vert A\vert$.
\end{claim}

\begin{claimproof}
Let $D$ be a minimum semitotal dominating set of $G$ with the maximum number of size-two components amongst all minimum semitotal dominating sets of $G$, such that $\vert B\cap D'\vert$ has minimum size amongst all minimum semitotal dominating sets with the maximum number of size-two components. Suppose for a contradiction that $\vert B\cap D'\vert\geq (k+1)(\vert A\vert +2)+\vert A\vert+1$. Then by Claim~\ref{OnlyOnePNInC}, every vertex in $B\cap D'$ has at most one private neighbour in $C$; and Claim~\ref{FewVerticesInCWithoutPN} ensures that there are at least $(k+1)(|A|+2)+2$ vertices $b_1,\ldots,b_{(k+1)(|A|+2)+2}\in B\cap D'$ which do have a private neighbour in $C$, say $c_1,\ldots,c_{(k+1)(|A|+2)+2}\in C$ respectively. Now observe that the set $S=(D\setminus \{b_{(k+1)(|A|+2)+2}\})\cup \{c_{(k+1)(|A|+2)+2}\}$ is a dominating set of $G$ of size $\gamma_{t2}(G)$ for otherwise $b_{(k+1)(|A|+2)+2}$ has a private neighbour $p\in N(b_{(k+1)(|A|+2)+2})$ which is not adjacent to $c_{(k+1)(|A|+2)+2}$, a contradiction to Claim~\ref{BCapDIsSmall}. However, $S$ cannot be a semitotal dominating set of $G$ as it would contradict the fact that $\vert B\cap D'\vert$ has minimum size amongst all minimum dominating sets with the maximum number of size-two components. Thus in $S$, either $c_{(k+1)(|A|+2)+2}$ has no witness or there exists a vertex $w\in D'$ which is at distance two from $b_{(k+1)(|A|+2)+2}$ but at distance at least three from $c_{(k+1)(|A|+2)+2}$ and every other vertex in $D'$. In the latter case, let $v$ be a common neighbour of $b_{(k+1)(|A|+2)+2}$ and $w$. Then, by assumption, $v$ has exactly two neighbours in $D'$, namely $b_{(k+1)(|A|+2)+2}$ and $w$, a contradiction to Claim~\ref{TwoneighboursInD'} as $\vert B\cap D'\vert\geq (k+1)(\vert A\vert +2)+\vert A\vert+1$ by assumption. Thus, assume that $c_{(k+1)(|A|+2)+2}$ does not have any vertices in $S$ at distance at most two. Now suppose that $N(c_{(k+1)(|A|+2)+2})$ contains two nonadjacent vertices, say $x$ and $y$. Then, since for any $i\in[|A|+2]$, the set $\{x,c_{(k+1)(|A|+2)+2},y,b_{(i-1)k+1},c_{(i-1)k+1},\ldots,b_{ik},c_{ik}\}$ cannot induce a $P_3+kP_2$, it follows that $x$ or $y$ has to be adjacent to $c_{(i-1)k+j}$ for some $j\in [k]$, say for every $i \in [|A|+2]$, $c_{ik}$ is adjacent to $x$ or $y$ without loss of generality. But then $(D\setminus \{b_{ik}~|~i\in[|A|+2]\})\cup \{x,y\}\cup A$ is a minimum semitotal dominating set containing a friendly triple, a contradiction to Theorem~\ref{thm:friendlytriple}(i). It follows that $N(c_{(k+1)(|A|+2)+2})$ is a clique and a similar reasoning shows that in fact $N(c_i)$ is a clique for every $i \in [(k+1)(|A|+2)+2]$. Now if there exist two indices $i,j\in [(k+1)(|A|+2)+2]$ such that $c_i$ and $c_j$ are at distance two then let $b$ be a common neighbour of $c_i$ and $c_j$. Then by Claim~\ref{BCapDIsSmall}, every private neighbour of $b_i$ besides $c_i$ has to be adjacent to $c_i$ and thus to $b$ as well; but then $(D\setminus \{b_i\})\cup \{b\}$ is a minimum semitotal dominating set containing more size-two components than $D$, a contradiction. If there exist two indices $i,j\in [(k+1)(|A|+2)+2]$ such that $c_i$ and $c_j$ are at distance three, then there are two adjacent vertices $v_i\in N(c_i)$ and $v_j\in N(c_j)$. Then for any $p \in \{i,j\}$, any private neighbour of $b_p$ besides $c_p$ has to be adjacent to $c_p$ by Claim~\ref{BCapDIsSmall} and thus to $v_p$; and any common neighbour of $b_i$ and $b_j$ must have at least one other neighbour in $D'$ by Claim~\ref{TwoneighboursInD'}. It follows that $S=(D \setminus \{b_i,b_j\}) \cup \{v_i,v_j\}$ is a dominating set of $G$ of size $\gamma_{t2}$. We further claim that $S$ is a semitotal dominating set of $G$. Indeed, suppose to the contrary that there exists a vertex $w \in D'$ such that $w$ has no witness in $S$. Then $w$ must be at distance two from $b_i$ or $b_j$, say $d_G(w,b_i)= 2$ without loss of generality. Let $v$ be a common neighbour of $w$ and $b_i$. Then there must exist at least two indices $p,q \in [(k+1)(|A|+2)+2] \setminus \{i\}$ such that $v$ is adjacent to $c_p$ and $c_q$ for otherwise $\{w,u,b_i\} \cup \{b_\ell,c_\ell~|~u \notin N(c_\ell)\}$ would contain an induced $P_3+kP_2$, a contradiction. But this implies in particular that $c_p$ and $c_q$ are at distance two, a contradiction by the previous case. Thus $S$ is a minimum semitotal dominating set of $G$; but $S$ contains strictly more size-two components than $D$, a contradiction. Thus for any $i,j \in [(k+1)(|A|+2)+2]$, $c_i$ and $c_j$ are at distance at least four and so for any $i \in [(k+1)(|A|+2)+2]$, $c_i$ is a regular vertex, a contradiction to our assumption. Thus $\vert B\cap D'\vert\leq (k+1)(\vert A\vert +2)+\vert A\vert$ and since $\vert D'\cap A\vert\leq\vert A\vert$, the claim follows.
\end{claimproof}

\begin{claim}\label{MostNeighbourhoodsAreCliques}
Let $D$ be a minimum semitotal dominating set of $G$ with a maximum number of size-two components amongst all minimum semitotal dominating set of $G$. If $S\subseteq C\cap D'$ is a subset of vertices which are pairwise at distance at least three and every vertex in $S$ has two nonadjacent neighbours then $\vert S\vert\leq k(k+1)-1$.
\end{claim}

\begin{claimproof}
Assume that $S'=\{c_1,\ldots,c_{k+1}\}\subseteq C\cap D'$ is a set of $k+1$ vertices which are pairwise at distance at least three and for every $i\in[k+1]$ there are two nonadjacent vertices $b_i, b'_i\in N(c_i)$. If for every $i,j\in[k+1]$ the vertices $c_i$ and $c_j$ are at distance at least four then $b_1,b'_1,c_1,\ldots,b_{k+1},b'_{k+1},c_{k+1}$ induce a $(k+1)P_3$, a contradiction. Hence, there exist two indices $i,j\in[k+1]$ such that $c_i$ and $c_j$ are at distance exactly three. 

Suppose for a contradiction that there is a set $S\subseteq C\cap D'$ of at least $k(k+1)$ vertices which are pairwise at distance at least three and such that for every vertex $v\in S$ there are two nonadjacent vertices in $N(v)$. By the above remark, there exist two vertices at distance exactly three in $S$. Let $S_1\subseteq N(S)$ be a maximum subset of $N(S)$ such that $G[S_1]$ contains exactly one edge and no two vertices in $S_1$ share a common neighbour in $S$. Observe that $\vert N(S_1)\cap S\vert=\vert S_1\vert$ and that $S_1\cup (N(S_1)\cap S)$ induces a $P_4+(\vert S_1\vert-2)P_2$. In particular, $\vert S_1\vert\leq k+1$. We construct a sequence of sets of vertices according to the following procedure.

\begin{itemize}
    \item[1.] Initialize $i=1$. Set $C_1=N(S_1)\cap S$ and $B_1=N(C_1)$.
    \item[2.] Increase $i$ by one.
    \item[3.] Let $S_i\subset N(S)\setminus B_{i-1}$ be a maximum set of vertices such that $G[S_i]$ contains exactly one edge and no two vertices in $S_i$ share a common neighbour in $S$. Set $C_i=C_{i-1}\cup (N(S_i)\cap S)$ and $B_i=B_{i-1}\cup N(C_i)$.
    \item[4.] If $\vert S_i\vert=\vert S_{i-1}\vert$, stop the procedure. Otherwise return to step 2.
\end{itemize}

Consider the value of $i$ at the end of the procedure (note that $i\geq 2$). Observe that since for any $j \in [i-1] \setminus \{1\}$, $|S_j| < |S_{j-1}|$ and $|S_1| \leq k + 1$, it follows that for any $j \in [i-1]$, $|S_j| \leq k + 2 - j$. Let us show that $|S_i| \geq 2$. Since for any $j \in [i-1]$, $|S_j| \leq k + 1$, we have that $|S \setminus C_j| = |S| - \sum_{p = 1}^j |S_p| \geq (k+1)^2 - j(k+1)$. Thus if $i \leq k+1$ then for any $j \in [i-1]$, $|S \setminus C_j| \geq k+1$ which implies by the above that $|S_j| \geq 2$ for any $j \in [i]$. We now claim that $i$ cannot be larger than $k+1$. Indeed, if $i > k+1$ then for any $j \in [k+1] \setminus {1}$, $|S_j| < |S_{j-1}|$ with $|S_{k+1}| \geq 2$ as shown previously; but $|S_j| \leq k + 2 - j$ for any $j \in [i-1]$ which implies that $|S_{k+1}| \leq 1$, a contradiction. Thus $i\leq k+1$ and so $|S_i| \geq 2$.

Now observe that for any vertex $c\in N(S_i)\cap S$, every neighbour $v\in N(c)$ has to be adjacent to $S_{i-1}$ as otherwise the procedure would have output $S_{i-1}\cup \{v\}$ instead of $S_{i-1}$. Furthermore, for any vertex $c\in N(S_{i-1})\cap S$ every neighbour $v\in N(c)$ has to be adjacent to a vertex in $S_{i}$ as otherwise the procedure would have output $S_i\cup \{v\}$ instead of $S_{i-1}$ (recall that $\vert S_i\vert=\vert S_{i-1}\vert$). It follows that $T=(D\setminus (N(S_i\cup S_{i-1})\cap S))\cup (S_i\cup S_{i-1})$ is a minimum semitotal dominating set of $G$: indeed, any vertex which is dominated or witnessed by a vertex in $N(S_i)\cap S$ or $N(S_{i-1})\cap S$ has to be dominated or witnessed by a vertex in $S_{i-1}\cup S_i$ by the observation above. But $T$ contains strictly more size-two components than $D$, a contradiction.
\end{claimproof}

\begin{claim}\label{AlmostAllAreDistanceThree}
Let $D$ be a minimum semitotal dominating set of $G$ with a maximum number of size-two components amongst all minimum semitotal dominating set of $G$. Then the number of vertices in $C\cap D'$ which are at distance two from another vertex in $C\cap D'$ is at most $2|A|+k(k+1)-3$.
\end{claim}

\begin{claimproof}
If every two vertices in $C\cap D'$ are at distance at least three from one another, then we are done. Thus assume that there are two vertices in $C\cap D'$ which are at distance two from one another. Let $\mathcal{S}=\arg\max_{S\subseteq B}\vert N(S)\cap C\cap D'\vert-\vert S\vert$ and let $S\in\mathcal{S}$ be a set of minimum size in $\mathcal{S}$. As there are two vertices in $C\cap D'$ which have a common neighbour, $S$ is non-empty. If $|N(S)\cap C\cap D'|\geq |A|+|S|$ then $D\setminus (N(S)\cap C\cap D')\cup S\cup A$ is a semitotal dominating set of $G$ which has cardinality at most $\vert D\vert$ and which contains a $P_3$, a contradiction to Theorem~\ref{thm:friendlytriple}(i). Hence, $\vert N(S)\cap C\cap D'\vert<\vert S\vert+ \vert A\vert$. We now claim that every vertex in $S$ is adjacent to two vertices in $C\cap D'$ which are adjacent to no other vertex in $S$. Indeed, if there exists a vertex $s\in S$ such that every one of its neighbour in $C\cap D'$ is adjacent to another vertex in $S$ then we could remove $s$ from $S$ without changing the cardinality of $\vert N(S)\cap C\cap D'\vert$, thereby contradicting the fact that $S \in \mathcal{S}$. If a vertex $s\in S$ has only one neighbour $c$ in $C\cap D'$ which is adjacent to no other vertex in $S$ then removing $s$ from $S$ would only remove $c$ from $N(S)\cap C\cap D'$, thus leaving the value of $|N(S)\cap C\cap D|-|S|$ unchanged while decreasing the cardinality of $S$, a contradiction to minimality of $\vert S\vert$. This implies that $\vert N(S)\cap C\cap D'\vert\geq 2\vert S\vert$. Combined with the inequality above, it follows that $|S|<|A|$ and $|N(S)\cap C\cap D'|\leq 2|A|-2$. Now denote by $C'=(C\cap D')\setminus N(S)$ the set of vertices in $C\cap D'$ which are not adjacent to a vertex in $S$. Note that every pair of vertices $c,c'\in C'$ does not have a common neighbour $b$ for otherwise $S'=S\cup \{b\}$ would be such that $\vert S'\vert=\vert S\vert+1$ and $\vert N(S')\cap C\cap D'\vert\geq \vert N(S)\cap C\cap D'\vert +2$ and thus $\vert N(S')\cap C\cap D'\vert-\vert S'\vert>\vert N(S)\cap C\cap D'\vert-\vert S\vert$, a contradiction to the choice of $S$. Hence, $C'$ is a set of vertices which are pairwise distance at least three and so by Claim~\ref{MostNeighbourhoodsAreCliques} it follows that at most $k(k+1)-1$ vertices in $C'$ do not have cliques as neighbourhoods. Denote $C''\subset C'$ the set of vertices whose neighbourhoods are cliques. Note that no vertex $c$ in $C'' \cap D'$ can be at distance two to any other vertex $c'$ in $C\cap D'$ for otherwise we could remove $c$ from $D'$ and replace it with a common neighbour of $c$ and $c'$ thus yielding a minimum semitotal dominating set containing strictly more size-two components than $D$, a contradiction to the choice of $D$. Thus, every vertex in $C\cap D'$ which has a common neighbour with another vertex in $C\cap D'$ must be contained in $N(S)\cap C\cap D'$ or in $C'\setminus C''$, which together have cardinality at most $2|A|+k(k+1)-3$.
\end{claimproof}

\begin{claim}\label{MostVerticesAreRegular}
There exists a minimum semitotal dominating set $D$ of $G$ such that $\vert D'\vert\leq (k+1)(\vert A\vert+2)+k(1+2(k+1))+4(\vert A\vert-1)$.
\end{claim}

\begin{claimproof}
It follows from Claim~\ref{MostVerticesInC} that there exists a minimum semitotal dominating set $D$ with the maximum number of size-two components amongst all minimum semitotal dominating sets of $G$ such that $\vert D'\setminus C\vert\leq (k+1)(\vert A\vert+2)+2\vert A\vert$. Let $C_1\subset C\cap D'$ be the set of vertices in $C\cap D'$ which are at distance at least three to every other vertex in $C\cap D'$. Let $C_2\subseteq C_1$ be the set of vertices in $C_1$ whose neighbourhood is a clique. Suppose for a contradiction that there are two vertices $c,c'\in C_2$ which are at distance three. Let $b\in N(c)$ and $b'\in N(c')$ be two adjacent vertices. Then $(D\setminus\{c,c'\})\cup\{b,b'\}$ is a minimum semitotal dominating set containing strictly more size-two components than $D$, a contradiction to the choice of $D$. Thus, the vertices in $C_2$ are pairwise at distance at least four from one another and so $\vert C_2\vert\leq k$ as $\mathcal{R}=\varnothing$. It now follows from Claim~\ref{AlmostAllAreDistanceThree} that $\vert (C\cap D')\setminus C_1\vert\leq 2\vert A\vert+ k(k+1)-3$ and from Claim~\ref{MostNeighbourhoodsAreCliques} that $\vert C_1\setminus C_2\vert\leq k(k+1)-1$, which implies the claim.
\end{claimproof}

Claim~\ref{clm:Rempty} now follows from Claims~\ref{ComponentsD} and \ref{MostVerticesAreRegular}.	

\end{document}